\documentclass[11pt,reqno]{amsart}
\usepackage{amssymb}
\usepackage{amsthm}
\usepackage[latin1]{inputenc}

\newcommand{\R}{\mathbb R}

\newcommand{\Z}{\mathbb Z}
\newcommand{\T}{\mathbb T}

\newcommand{\hil}{\mathcal{H}}
\newcommand{\sgn}{\text{sgn}}

\newcommand{\jd}{\rangle}

\newcommand{\ji}{\langle}

\numberwithin{equation}{section}
\newtheorem{theorem}{Theorem}[section]
\newtheorem{proposition}[theorem]{Proposition}
\newtheorem{remark}[theorem]{Remark}
\newtheorem*{remarksn}{Remark}
\newtheorem*{remarks}{Remarks}
\newtheorem{lemma}[theorem]{Lemma}

\newtheorem*{TA}{Theorem A}
\newtheorem*{TB}{Theorem B}
\newtheorem*{TC}{Theorem C}
\newtheorem*{claim1}{CLAIM 1}
\newtheorem*{claim2}{CLAIM 2}
\begin{document}
\title[The  dispersion generalized B-O equation]{The IVP for the dispersion generalized  Benjamin-Ono equation in weighted Sobolev spaces}
\author{Germ\'an Fonseca}
\address[G. Fonseca]{Departamento  de Matem\'aticas\\
Universidad Nacional de Colombia\\ Bogota\\Colombia}
\email{gefonsecab@unal.edu.co}
\author{Felipe Linares}
\address[F. Linares]{IMPA\\ Rio de Janeiro\\Brazil}
\email{linares@impa.br}
\thanks{The second author was supported by CNPq and FAPERJ/Brasil}
\author{Gustavo Ponce}
\address[G. Ponce]{Department  of Mathematics\\
University of California\\
Santa Barbara, CA 93106\\
USA.} \email{ponce@math.ucsb.edu}
\thanks{The third author is supported  by  the NSF  DMS-1101499}
\begin{abstract} We study the initial value problem associated to the dispersion generalized  Benjamin-Ono equation.
Our aim is to establish  persistence results in weighted Sobolev spaces and to deduce from them some sharp unique
continuation properties of solutions to this equation. In particular, we shall establish optimal decay rate for the  solutions
of this model.
\end{abstract}

\maketitle

\section{Introduction}

This work is concerned with the initial value problem (IVP)
for the   dispersion generalized Benjamin-Ono (DGBO) equation
\begin{equation}\label{DGBO}
\begin{cases}
\partial_t u +   D^{1+a}\partial_x u +u\partial_x u = 0, \qquad t, x\in \R,\;\; 0<a<1,\\
u(x,0) = u_0(x),
\end{cases}
\end{equation}
where $D^s$ denotes  the homogeneous derivative of order $s\in\R$,
$$
D^s=(-\Delta)^{s/2}\;\;\;\text{so}\;\;\;D^s f=c_s\big(|\xi|^s\widehat{f}\,\big)^{\vee}, \;\;\;\text{with} \;\;\;D^s=(\mathcal H\,\partial_x)^s\;\;\;\text{if}\;\;\;n=1,
$$
where $\hil$ denotes the Hilbert transform,
\begin{equation*}
\begin{split}
\hil f(x)&=\frac{1}{\pi} {\rm p.v.}\big(\frac{1}{x}\ast f\big)(x)\\
&=\frac{1}{\pi}\lim_{\epsilon\downarrow 0}\int\limits_{|y|\ge \epsilon} \frac{f(x-y)}{y}\,dy=(-i\,\sgn(\xi) \widehat{f}(\xi))^{\vee}(x).
\end{split}
\end{equation*}
These equations model vorticity waves in the coastal zone, see \cite{MoSaTz} and references therein.

When $a=1$  the equation in \eqref{DGBO} becomes  the famous Korteweg-de Vries (KdV) equation
\begin{equation}\label{KdV}
\partial_t u- \partial^3_x u +u\partial_x u = 0, \qquad t, x
\in \R,
\end{equation}
and when  $a=0$ the equation in \eqref{DGBO} agrees with the well known Benjamin-Ono (BO) equation
\begin{equation}\label{BO}
\partial_t u +   \mathcal{H}\partial^2_x u +u\partial_x u = 0, \qquad t, x
\in \R.
\end{equation}

Both the  KdV  and the BO  equations originally arise as models in one-dimensional waves propagation
(see \cite{KdV}, \cite{Be}, and \cite{On})
and have widely been studied in many different contexts.  They  present
several similarities: both possess infinite conserved quantities, define Hamiltonian systems, have
multi soliton solutions and  are completely integrable. The local well-posedness (LWP) and global well-posedness (GWP) of their  associated IVP in
the classical Sobolev spaces $H^s(\R),\;s\in\R$ have been extensively investigated.

In the case of the KdV  equation this problem has been studied in  \cite{Sa}, \cite{BoSm}, \cite{Ka}, \cite{KePoVe}, \cite{Bo}, \cite{KePoVe1}, \cite{CoKeStTaT},
and finally \cite{Gu} where global well-posedness was established for $s\geq -3/4$.

 In the case of the BO equation the same well-posedness problem has been considered in  \cite{Sa}, \cite{ABFS}, \cite{Io1}, \cite{Po}, \cite{KoTz1}, \cite{KeKo}, \cite{Ta},
 \cite{MoRi},  \cite{BuPl},
and  \cite{IoKe} where global well-posedness was established for $s\geq 0$ (for further discussion  we refer to  \cite{MoPi}).

  However, there are two remarkable differences between the existence theory for these two models.
The first  is the  fact that  one can give a local existence theory for the IVP associated to the KdV in $H^s(\R)$ based only on  the contraction principle. This can not
 be done in the case of the BO. This is a consequence of the lack of smoothness of the application data-solution in the
 BO setting established in  \cite{MoSaTz}. There it was proved  that this map is
  not locally $C^2$. Actually, in \cite{KoTz2} it was proved that this map is  not even locally
  uniformly continuous.

The second remarkable difference between these equations is concerned with the persistent property  of the solutions  (i.e. if the data $u_0\in X$, a function space, then the corresponding solution $u(\cdot)$ describes  a continuous curve in $X$, $u\in C([-T,T]:X)$, $\,T>0$)
 in weighted Sobolev spaces.
In \cite{Ka} it was shown that the KdV flow preserves the Schwartz class. However, it was first established by Iorio   \cite{Io1} and \cite{Io2} that in
general,  polynomial type decay is not preserved by the BO
flow. The results in \cite{Io1}, \cite{Io2} were recently extended to fractional order
weighted Sobolev spaces in \cite{FoPo}. In order to present  these
results, we  introduce the weighted Sobolev spaces
\begin{equation}
\label{spaceZ}
Z_{s,r} = H^s(\R) \cap L^2(|x|^{2r}dx),\,\,\, s, r \in \R ,
\end{equation}
and
\begin{equation}
\label{spaceZdot}
\dot Z_{s,r}=\{ f\in H^s(\R)\cap L^2(|x|^{2r}dx)\,:\,\widehat {f}(0)=0\},\;\;\;\;\;\;\;s,\,r\in\R.
\end{equation}
The  well-posedness results for the IVP associated to the BO equation in weighted Sobolev spaces  can be stated as:
\begin{TA}$($\cite{FoPo}$)$\label{theorem3}
\begin{itemize}
\item[(i)] Let $s\geq 1, \;r\in [0,s]$, and $\,r<5/2$. If $u_0\in Z_{s,r}$, then the solution $\,u\,$ of the IVP associated to the BO equation \eqref{BO} satisfies that
$$
u\in C([0,\infty):Z_{s,r}).
$$

\item[(ii)] For  $s>9/8$  ($s\geq 3/2$), $\;r\in [0,s]$, and $\,r<5/2$ the IVP associated to the BO equation \eqref{BO} is LWP (GWP resp.) in $Z_{s,r}$.

\item[(iii)] If  $\,r\in [5/2,7/2)$ and $\,r\leq s$, then the IVP \eqref{BO} is GWP
in $\dot Z_{s,r}$.
\end{itemize}

\end{TA}

\begin{TB}$($\cite{FoPo}$)$\label{theorem4}
 Let $u\in C([0,T] : Z_{2,2})$ be a solution of the IVP \eqref{BO}. If   there exist  two different times
 $\,t_1, t_2\in [0,T]$ such that
 \begin{equation}
 \label{2timesw}
 u(\cdot,t_j)\in Z_{5/2,5/2},\;j=1,2,\;\text{then}\;\;\widehat {u}_0(0)=0,\;\,(\text{so}\;\, u(\cdot, t)\in  \dot Z_{5/2,5/2}).
 \end{equation}

\end{TB}

 \vskip.1in

 \begin{TC}$($\cite{FoPo}$)$\label{theorem5}
 Let $u\in C([0,T] : \dot Z_{3,3})$ be a solution of the IVP \eqref{BO}. If   there exist  three different times
 $\,t_1, t_2, t_3\in [0,T]$ such that
 \begin{equation}
 \label{3timeus}
 u(\cdot,t_j)\in \dot Z_{7/2,7/2},\;\;j=1,2,3,\;\;\text{then}\;\;\;u(x,t)\equiv 0.
 \end{equation}
\end{TC}
\vskip.1in

We point out that Iorio's results correspond to  the indexes $s\geq r=2$ in Theorem A part (ii), $s\geq r=3$  in Theorem A part (iii)
and $s\geq r=4$ in Theorem C.

Regarding the DGBO equation \eqref{DGBO}, we notice that for $\,a\in(0,1)\,$ the dispersive effect is stronger than  the one
for the BO equation but  still too weak compared to that of the KdV equation. Indeed it was shown
in \cite{MoSaTz} that for the IVP associated to the DGBO equation \eqref{DGBO} the flow map data-solution
from $H^s(\R)$ to $C([0,T]\,:\,H^s(\R))$ fails to be locally $C^2$ at the
origin for any $T>0$ and any $s\in\R$ as in the case of the BO equation. Therefore, so far  local well-posedness
in classical Sobolev spaces $H^s(\R)$ for \eqref{DGBO} cannot be obtained by an
argument based only on the contraction principle.  Local well-posedness in classical Sobolev spaces for \eqref{DGBO}
has been studied in  \cite{KePoVe},  \cite{He}, \cite{MoRi2}, \cite{Gu2}, and \cite{HIKK} where local well-posedness was established for $s\geq 0$.

  Real solutions of the IVP \eqref{DGBO} satisfy  at least three conserved quantities:
\begin{equation}
\begin{aligned}
\label{laws}
&\;I_1(u)=\int_{-\infty}^{\infty}u(x,t)dx,\;\;\;\;I_2(u)=\int_{-\infty}^{\infty}u^2(x,t)dx,\\
&\;I_3(u)=\int_{-\infty}^{\infty}\,(|D^{\frac{1+a}{2}}u|^2+\frac
{u^3}{6})(x,t)dx.
\end{aligned}
\end{equation}

In particular, we have that the local results in \cite{HIKK} extend globally in time.

Concerning the form of the traveling wave solution of \eqref{DGBO} it is convenient to consider
$$
 v(x,t)=-u(x,-t),
 $$
 where $u(x,t)$ satisfies equation \eqref{DGBO}. Thus,
  \begin{equation}
  \label{BOneg}
  \partial_t v -D^{1+a}\partial_x v +v\partial_x v = 0, \qquad t, x\in\R,\;\;\;\;\;0\leq a\leq1,
  \end{equation}
Traveling wave solutions of  \eqref{BOneg} are solutions of the form
$$
v(x,t)=c^{1+a}\,\phi_a(c(x-c^{1+a}\,t)),\;\;\;\;\;\;c>0,
$$
where $\,\phi_a\,$ is called the ground state, which is an even, positive, decreasing (for $x>0$) function.
In the case of the KdV equation ($a=1$ in \eqref{BOneg}) one has that
$$
\phi_1(x)=\,\frac{3}{2}\,\text{sech}^2\left(\frac{x}{2}\right),
$$
whose uniqueness follows by elliptic theory.

In the case of the BO equation ($a=0$ in  \eqref{BOneg})
one has that
\begin{equation}
\label{AA1}
\phi_0(x)=\frac{4}{1+x^2},
\end{equation}
whose uniqueness (up to symmetry of the equation)  was established in \cite{CAJT}.

In the case $a\in (0,1)$ in \eqref{BOneg} the existence of the ground state was established in \cite{We}
by variational arguments. Recently, uniqueness of the ground state for $a\in (0,1)$  was established in \cite{RFEZ}. However, no explicit formula
is know for $\,\phi_a, \,a\in (0,1)$. In \cite{MKR} the following upper bound for the decay of the ground state
was deduced
$$
\phi_a(x)\leq \frac{c_a}{(1+x^2)^{1+a/2}},\hskip15pt 0<a<1.
$$

Thus, one has that for $a\in [0,1)$ the ground state has a very mild decay in comparison with that for the KdV equation $a=1$.
Roughly speaking, this is a consequence of the non-smoothness of the symbol modeling the dispersive relation in \eqref{DGBO} $\,\sigma_a(\xi)=
|\xi|^{1+a}\,\xi$.

Our  goal in this work is to extend the results in Theorems A-C for the DGBO equation \eqref{DGBO},
by proving persistent properties of solution of \eqref{DGBO}
in the  weighted Sobolev spaces \eqref{spaceZ}. This will lead us to obtain some optimal uniqueness properties of solutions of this equation
as well as  to establish what is the maximum rate of decay of a solution of \eqref{DGBO}.

In order to motivate our results  we first recall the fact  that for dispersive equations the decay of the data is preserved by the solution only if they have enough regularity.
More precisely,  persistence property of the solution $u=u(x,t)$ of the IVP \eqref{DGBO}
 in the weighted Sobolev spaces $Z_{s,\,r}$
can only hold if $\,s\geq (1+a)r$.
This can be seen from the fact that the linear part of the equation \eqref{DGBO}
\begin{equation}
\label{commu}
L=\partial_t+D^{1+a}\partial_x\;\;\;\text{commutes with }\;\;\;\;\Gamma= x-(a+2)tD^{1+a}.
\end{equation}
 Hence,
it is natural to consider well-posedness in the weighted Sobolev spaces $Z_{s,\,r},\,s\geq (1+a)r$.

Let us state our main results:

\begin{theorem}\label{lwp-zsr}
{\rm (a)} Let $a\in (0,1)$. If $u_0\in Z_{s,r}$, then the solution $\,u\,$ of
the IVP \eqref{DGBO} satisfies $u\in C([-T,T]\,:\,Z_{s,r})$ if either
\begin{enumerate}
\item[(i)] $s\geq (1+a)$ and  $r\in(0,1]$,

or

\item[(ii)] $s\geq 2(1+a)$ and  $r\in (1,2]$,

or

\item[(iii)] $s\geq [(r+1)^-](1+a)$ and $\;2<r<5/2+a$, with $[\cdot]$ denoting the integer part function.
\end{enumerate}
{\rm(b)}  If $u_0\in \dot{Z}_{s,r}$, then the solution $u$ of the IVP
\eqref{DGBO} satisfies
$$
u\in C([-T,T]\,:\,\dot{Z}_{s,r}),
$$
 whenever
\begin{enumerate}
\item[(iv)] $s\geq [(r+1)^-](1+a)\;$ and $\;5/2+a\le r<7/2+a$.
\end{enumerate}
\end{theorem}


\begin{theorem}\label{theorem7}
 Let $u\in C([-T,T] : Z_{s, (5/2+a)^-})$ with
 $$
 T>0\;\;\;\;\;\text{and}\;\;\;\;\,s\geq (1+a)(5/2+a)+(1-a)/2
 $$
  be a solution of the
 IVP \eqref{DGBO}. If   there exist  two times $\,t_1, t_2\in [-T,T]$, $t_1\neq t_2$,  such that
\begin{equation}\label{2timesmean}
u(\cdot,t_j)\in Z_{s,5/2+a},\;\;j=1,2.
\end{equation}
Then
\begin{equation}\label{a4}
\widehat{u}(0,t)\!=\!\int u(x,t)dx=\!\int u_0(x)\,dx\!=\widehat{u}_0(0)=0\hskip8pt\text{for all}\hskip8pt t\in[-T,T].
\end{equation}
\end{theorem}

 \begin {remarks} \hskip10pt
\begin{enumerate}
\item[a)]
 Theorem \ref{theorem7}  shows that persistence in $Z_{s,r}$  with $r=(5/2+a)^-$ is the best possible for
 general initial data. In fact, it shows that for data $u_0\in Z_{s, r}, \,s\geq (1+a)r+(1-a)/2$, $r\geq 5/2+a$ with
 $\, \widehat{u}_0(0)\neq 0$ the corresponding solution
 $u=u(x,t)$ verifies that
 $$
|x|^{ (5/2+a)^-} u\in L^{\infty}([0,T] :L^2(\R)),\;\;\;\;\;T>0,
$$
but there does not exist  a non-trivial solution $u$ corresponding to data $u_0$ with  $\,\widehat{u}_0(0)\neq 0$ such that
$$
|x|^{ 5/2+a} u\in L^{\infty}([0,T'] :L^2(\R)),\;\;\;\;\;\text{for some}\;\;\;\;T'>0.
$$

(

 \item[b)] The result in Theorem \ref{lwp-zsr} for $s=1+a$ was established in \cite {CoKeSt}.
 \end{enumerate}
   \end{remarks}

\begin{theorem}\label{theorem8}
Let $u\in C([-T,T] : Z_{s, (7/2+a)^-})$ with
$$
T>0 \hskip15pt\text{and}\hskip 15pt s\geq (1+a)(7/2+a)+\frac{1-a}{2},
$$
 be a solution of
the IVP \eqref{DGBO}. If   there exist  three different times
 $\,t_1, t_2, t_3\in [-T,T]$ such that
 \begin{equation} \label{3timesvanish}
 u(\cdot,t_j)\in \dot{Z}_{s,7/2+a},\;\;j=1,2, 3.
 \end{equation}
 Then
 \begin{equation*}
 u\equiv 0.
 \end{equation*}
\end{theorem}

\begin{remarks}\hskip10pt
\begin{enumerate}
\item[a)] Theorem \ref{theorem8} shows that the decay $r=(7/2+a)^-$ is
the largest possible. More precisely, Theorem \ref{lwp-zsr} part (b) tells us that there are non trivial solutions $u=u(x,t)$ verifying
 $$
|x|^{ (7/2+a)^-} u\in L^{\infty}([0,T] :L^2(\R)),\;\;\;\;\;T>0,
$$
and Theorem \ref{theorem8} guarantees that there does not exist a non-trivial solution such that
$$
|x|^{ 7/2+a} u\in L^{\infty}([0,T'] :L^2(\R)),\;\;\;\;\;\text{for some}\;\;\;\;T'>0.
$$
\item[b)] We shall prove this result in the most general case $s=(1+a)(7/2+a)+\frac{1-a}{2}.$ Also, we
will carry out the details in the case $a\in [1/2,1).$ It will be clear from our argument how to extend
the result to the case $a\in (0,1/2).$
\end{enumerate}
\end{remarks}

\begin{theorem}\label{theorem9}
Let $u\in C([-T,T] : Z_{s, (7/2+a)^-})$ with
$$
 T>0\;\;\;\;\;\text{and}\;\;\;\;\;s\geq (1+a)(7/2+a)+(1-a)/2,
 $$
  be a solution
of the IVP \eqref{DGBO}. If there exist $\,t_1, t_2\in [-T,T]$, $t_1\neq t_2$ such that
\begin{equation}\label{2timesvanish}
\begin{split}
u(\cdot,t_j)\in \dot{Z}_{s,7/2+a},\;\;j=1,2,\\
\intertext{and}
\hskip2cm\int xu(x,t_1) dx=0\;\;\,\,\,\text{or}\;\;\,\,\,\, \int xu(x,t_2) dx=0,
\end{split}
\end{equation}
then
\begin{equation*}
u\equiv 0.
\end{equation*}
\end{theorem}

\begin{remarksn}
Theorem \ref{theorem9} tells us that the conditions of Theorem \ref{theorem8} can be reduced to
two times provided the first momentum of the solution $u$ vanishes at one of them.
\end{remarksn}

\begin{theorem}\label{theorem10}
Let $u\in C([-T,T] : Z_{s, (7/2+a)^-})$ with
$$
T>0\;\;\;\;\;\;\text{and}\;\;\;\;\;s\geq (1+a)(7/2+[1+2a]/2)+(1-a)/2
$$
 be a non-trivial solution of
the IVP \eqref{DGBO} such that
\begin{equation}
u_0\in \dot Z_{s,\frac72 +\tilde a},\,\, \tilde a=[1+2a]/2, \;\;\;\hskip15pt \text{and}\;\;\;\;\;\hskip15pt \int_{-\infty}^{\infty} {x\,u_0(x)\,dx}\neq 0.
 \end{equation}
 Then there exists $t^*\neq0$ with
 \begin{equation}
 t^*=-\frac{4}{\|u_0\|_2^2}\int_{-\infty}^{\infty} {x\,u_0(x)\,dx},
 \end{equation}
 such that $u(t^*)\in\dot Z_{s,\frac72+\tilde a}.$
 \end{theorem}

\begin{remarks}\hskip20pt
 \begin{enumerate}
\item[a)] Notice that $\,\tilde a>a$, so Theorem \ref{theorem10} shows that the condition of Theorem \ref{theorem8}
 at two times is in general not sufficient to guarantee that $u\equiv 0$.  So, in this regard
 Theorem \ref{theorem9} is optimal.

\item[b)]  The results in Theorem \ref{theorem8} and Theorem \ref{theorem10} present  an striking difference with other unique continuation properties deduced for other dispersive models.
Using the information at two different times, uniqueness results have been established for the generalized KdV equation in \cite{EKPV1}, for the semi-linear Schr\"odinger equation in \cite{EKPV2}, and  for the Camassa-Holm model  in \cite{HMPZ}.  Theorem \ref{theorem10}
affirms that the uniqueness condition with the weight $|x|^{7/2+a}$ does not hold at two different
times but Theorem \ref{theorem8} guarantees that it does  at three times. Similar result for the Benjamin-Ono equation ($a=0$ in \eqref{DGBO})
was obtained in \cite{GFFLGP}.
\end{enumerate}
\end{remarks}

\vskip.in
One can consider the IVP \eqref{DGBO} with $a>1$. In this case our results still hold, with the
appropriate modification in the well-posedness in $H^s(\R)$, if $a$ is not an odd integer. In
the case where $a$ is an odd integer, one has solutions with exponential decay as in the
case of the KdV equation ($a=1$ in \eqref{DGBO}).

Finally, we consider the generalization of the IVP \eqref{DGBO} to higher nonlinearity
\begin{equation}\label{k-dgbo}
\begin{cases}
\partial_t u+D^{1+a}\partial_x u+u^k\partial_xu=0,\hskip10ptt,x\in\R,\;k\in\Z^{+},\\
u(x,0)=u_0(x).
\end{cases}
\end{equation}

In this case our positive results, Theorems \ref{lwp-zsr} -- \ref{theorem7}, still hold (with the
appropriate modification in the well-posedness in $H^s(\R)$). Our unique continuation results
(Theorems \ref{theorem8}--\ref{theorem9}) can be extended to the case where $k$ in
\eqref{k-dgbo} is odd. In this case one has that the time evolution of the first momentum of
the solution is given by the formula
\begin{equation*}
\int_{-\infty}^{\infty} x\,u(x,t)\,dx=\int_{-\infty}^{\infty} x\,u_0(x)\,dx+\frac{1}{k+1}
\int_0^t \int_{-\infty}^{\infty} u^{k+1}(x,t')\,dx',dt'.
\end{equation*}
Thus, it is an increasing function. Hence, defining $t^{*}\neq 0$ as the solution of the
equation
\begin{equation}\label{g-dgbo-2}
\int_0^{t^{*}} \int_{-\infty}^{\infty} x\,u(x,t)\,dxdt=0,
\end{equation}
one sees that there is at most one solution of \eqref{g-dgbo-2} but its existence it
is not guaranteed. So the statements in Theorems \ref{theorem8}--\ref{theorem9}
would have to be modified accordingly to this fact.

 The rest of this paper is organized as follows: section 2 contains some preliminary estimates
 to be used in  the coming sections. Section 3 contains the proof of Theorem \ref{lwp-zsr}.  
Theorem \ref{theorem7}, Theorem \ref{theorem8}, Theorem \ref{theorem9}, and Theorem \ref{theorem10} will be proven in sections 4, 5, 6, and 7
respectively.

\section{Preliminary Estimates}

We begin this section by introducing the notation needed in this work.
We use $\|\cdot\|_{L^p}$ to denote the $L^p(\R)$ norm. If necessary, we use subscript to inform which
variable we are concerned with. The mixed norm $L^q_tL^r_x$ of
$f=f(x,t)$ is defined as
\begin{equation*}
\|f\|_{L^q_tL^r_x}= \left(\int \|f(\cdot,t)\|_{L^r_x}^q dt
\right)^{1/q},
\end{equation*}
with the usual modifications when $q =\infty$ or $r=\infty$. The
$L^r_xL^q_t$ norm is similarly defined.

We define the spatial Fourier transform of $f(x)$ by
\begin{equation*}
\hat{f}(\xi)=\int_{\R} e^{-ix\xi}f(x)\,dx.
\end{equation*}

We shall also define $J^s$ to be the Fourier multiplier with symbol $\ji \xi \jd^s = (1+|\xi|^2)^\frac{s}{2}$.
Thus, the norm in the Sobolev space $H^s(\R)$ is given by
\begin{equation*}
\|f\|_{s,2}\equiv \|J^s f\|_{L^2_x}=\|\ji \xi\jd^s\widehat{f}\,\|_{L^2_{\xi}}.
\end{equation*}

A function $\chi\in C^{\infty}_0$, supp $\chi\subseteq [-2,2]$ and $\chi\equiv 1$
in $(-1,1)$ will appear several times in our arguments.

 For $a\in (0,1)$ fixed we introduce $F_{j}$'s as
being
\begin{equation}\label{notation2}
F_j(t,\xi,\widehat{u}_0)=\partial_{\xi}^j (e^{-it|\xi|^{1+a}\xi}\widehat{u}_0(\xi)),
\end{equation}
for $j=0,1,2,3,4.$ Thus
\begin{equation*}
\begin{split}
F_1(t,\xi,\widehat{u}_0)&=-(2+a)it|\xi|^{1+a}e^{-it|\xi|^{1+a}\xi}\widehat{u}_0(\xi)
+e^{-it|\xi|^{1+a}\xi}\partial_{\xi}\widehat{u}_0(\xi),
\end{split}
\end{equation*}
\begin{equation*}
\begin{split}
 F_2(t,\xi,\widehat{u}_0)=&\,
e^{-it|\xi|^{1+a}\xi}(-it(2+a)(1+a)|\xi|^a\text{sgn}(\xi)\widehat{u}_0(\xi)\\
&-(2+a)^2t^2|\xi|^{2(a+1)}\widehat{u}_0(\xi)\\
&-2it(2+a)|\xi|^{1+a}\partial_{\xi}\widehat{u}_0(\xi)+\partial_{\xi}^2\widehat{u}_0(\xi))\\
=&\,(B_1+B_2+B_3+B_4)(t,\xi,\widehat{u}_0),
\end{split}
\end{equation*}
\begin{equation*}
\begin{split}
F_3(t,\xi,\widehat{u}_0)=&\, e^{-it|\xi|^{1+a}\xi}\big(-ita(1+a)(2+a)|\xi|^{a-1}\widehat{u}_0(\xi)\\
&-3t^{2}(2+a)^{2}(1+a)|\xi|^{2a+1}\,\text{sgn}(\xi)\widehat{u}_0(\xi)\\
&+it^3(2+a)^{3}|\xi|^{3(1+a)}\widehat{u}_0\\
&-3it(2+a)(1+a)\,|\xi|^{a}\text{sgn}(\xi)\partial_{\xi}\widehat{u}_0(\xi)\\
&-3t^{2}(2+a)^{2}|\xi|^{2(1+a)}\partial_{\xi}\widehat{u}_0(\xi)\\
&-3it(2+a)|\xi|^{1+a}\partial_{\xi}^2\widehat{u}_0(\xi)+\partial_{\xi}^3\widehat{u}_0(\xi)\big)\\
=&\,(D_1+D_2+D_3+D_4+D_5+D_7)(t,\xi,\widehat{u}_0),
\end{split}
\end{equation*}
\begin{equation*}
\begin{split}
F_4(t,\xi,\widehat{u}_0)=&\, e^{-it|\xi|^{1+a}\xi}\big(-it(2+a)(1+a)a(a-1)|\xi|^{a-2}\,\text{sgn}(\xi)\widehat{u}_0(\xi)\\
&-t^2(2+a)^2(1+a)(7a+3)|\xi|^{2a}\widehat{u}_0(\xi)\\
&+6it^3(2+a)^{3}(1+a)|\xi|^{3a+2}\text{sgn}(\xi)\widehat{u}_0(\xi)\\
&+t^4(2+a)^{4}|\xi|^{4(a+1)}(\xi)\widehat{u}_0(\xi)\\
&-4ita(1+a)(2+a)|\xi|^{a-1}\partial_{\xi}\widehat{u}_0(\xi)\\
&-12t^{2}(2+a)^{2}(1+a)|\xi|^{2a+1}\,\text{sgn}(\xi)\partial_{\xi}\widehat{u}_0(\xi)\\
&+4it^3(2+a)^{3}|\xi|^{3(1+a)}\partial_{\xi}\widehat{u}_0
-6t^{2}(2+a)^{2}|\xi|^{2(1+a)}\partial_{\xi}^2\widehat{u}_0(\xi)\\
&-6it(2+a)(1+a)\,|\xi|^{a}\text{sgn}(\xi)\partial_{\xi}^2\widehat{u}_0(\xi)\\
&-4it(2+a)|\xi|^{1+a}\partial_{\xi}^3\widehat{u}_0(\xi)+\partial_{\xi}^4\widehat{u}_0(\xi)\big)\\
=&\,(E_1+\cdots+E_{11})(t,\xi,\widehat{u}_0).
\end{split}
\end{equation*}
\vskip3mm

The next two results will be essential in the analysis below.

The first one is an extension of the Calder\'on commutator theorem \cite{Ca} found in
 \cite{DaMcPo}.

\begin{lemma}\label{dmp1} Let $\hil$ denote
the Hilbert transform. Then for any $p\in (1,\infty)$ and any $l, m
\in \Z^{+}\cup\{0\}$ there exists $c = c(p; l; m) > 0$ such that
\begin{equation}\label{c-dmp1}
\|\partial_x^l [\hil;\psi]\partial_x^m f\|_{L^p}\le
c\,\|\partial_x^{m+l}\psi\|_{L^{\infty}}\|f\|_{L^p}.
\end{equation}
\end{lemma}

See Lemma 3.1 in \cite{DaMcPo}.

\begin{proposition}\label{dmp2}
Let $\alpha \in [0, 1)$, $\beta\in(0, 1)$ with $\alpha+\beta \in [0,
1]$. Then for any $p, q \in  (1,\infty)$ and for any $\delta > 1/q$
there exists $c = c(\alpha;\beta; p; q;\delta) > 0$ such that
\begin{equation}\label{c-dmp2}
\|D^{\alpha} [D^{\beta}; \psi]D^{1-(\alpha+\beta)} f\|_{L^p}\le
c\,\|J^{\delta}\partial_x \psi\|_{L^q} \|f\|_{L^p},
\end{equation}
where $J := (1 -\partial_x^2)^{1/2}$.
\end{proposition}

See Proposition 3.2 in \cite{DaMcPo}.

\vspace{1cm}

Using the notation
\begin{equation}\label{group}
W_a(t)f=\big(e^{-it|\xi|^{1+a}\xi}\, \widehat{f}\,\big)^{\vee}
\end{equation}
we recall the following linear estimates:

\begin{proposition}\label{prop3a}(Smoothing Effects and Maximal Function)
\begin{enumerate}
\item Homogeneous.
\begin{equation}\label{homoge}
\| D^{(1+a)/2}W_a(t)f\|_{L_x^{\infty}L_T^2}\le c_a\,\|f\|_2.
\end{equation}
\item Nonhomogeneous and Duality.
\begin{equation}\label{nhplusdual}
\begin{split}
&\| D^{s+a/2+1/2}\!\! \int_0^t\!\! W_a(t-t') F(t')d t'
\|_{L_x^{\infty}L_T^2}\!+\!
\| D^{s}\!\! \int_0^t \!\!W_a(t-t') F(t')d t' \|_{L_T^{\infty}L_x^2}\\
&\le T^{a/2} \| D^{s-1/2+a/2}F\|_{L_x^{2/(2-a)} L_T^2}.
\end{split}
\end{equation}
\item Maximal function estimate
\begin{equation}\label{maximal}
\|W_a(t)f\|_{L_x^2L_T^{\infty}}\le c\,(1+T)^{\rho}\,\|f\|_{s,2}
\end{equation}
where $\rho>3/4$ and $s>(2+a)/4$.
\end{enumerate}
\end{proposition}

\begin{proof} For the proof of inequalities \eqref{homoge} and \eqref{maximal} see
\cite{KePoVe}. The inequality \eqref{nhplusdual} follows by
interpolation.
\end{proof}

\begin{proposition}\label{prop*}\hskip20pt \begin{enumerate}
\item[(i)] Given $\phi\in L^{\infty}(\R)$, with $\partial_x^{\alpha}\phi\in L^2(\R)$ for $\alpha=1,2$, then
for any $\theta\in(0,1)$
\begin{equation}\label{*1}
\|[J^{\theta};\phi]f\|_2\le c_{\theta,\phi}\,\|f\|_2.
\end{equation}
\item[(ii)] If $\eta\in(0,1]$, then
\begin{equation}\label{*1b}
\|J^{\eta}(fg)-fJ^{\eta}g\|_2\le c\,\|\widehat{\partial_x f}\|_1\|g\|_2.
\end{equation}
\end{enumerate}
\end{proposition}

\begin{proof} We first prove \eqref{*1}. Since
\begin{equation*}
\begin{split}
\big([J^{\theta};\phi] f\big)^{\wedge}(\xi)&=(J^{\theta}(\phi f)-\phi J^{\theta}f\big)^{\wedge}(\xi)\\
&=\int \Big((1+\xi^2)^{\theta/2} -(1+\eta^2)^{\theta/2}\Big)\widehat{\phi}(\xi-\eta)\widehat{f}(\eta)d\eta,
\end{split}
\end{equation*}
the mean value theorem leads to
\begin{equation*}
\big|\big([J^{\theta};\phi]f\big)^{\wedge}(\xi)\big|
\le c_{\theta}\,\int |\xi-\eta| |\widehat{\phi}(\xi-\eta)||\widehat{f}(\eta)|\,d\eta
=c_{\theta} \big(|\widehat{\partial_x\phi}|\ast |\widehat{f}\,|\big)(\xi).
\end{equation*}
Then by Young's inequality
\begin{equation*}
\begin{split}
\|[J^{\theta};\phi]f\|_2&\le c_{\theta} \| |\widehat{\partial_x\phi}|\ast |\widehat{f}|\|_2
\le c_{\theta}\|\widehat{\partial_x\phi}\|_1\|\widehat{f}\,\|_2\\
&\le c_{\theta}\|\partial_x\phi\|_{1,2}\|f\|_2\le c_{\theta,\phi}\|f\|_2.
\end{split}
\end{equation*}

To show \eqref{*1b} we notice that
\begin{equation*}
\begin{split}
\|J^{\eta}(fg)-fJ^{\eta}g\|_2&\le \|\int |(1+|\xi|^2)^{\eta/2}-(1+|\zeta|^2)^{\eta/2}|\widehat{f}(\xi-\zeta)||\widehat{g}(\zeta)|d\zeta\|_2\\
&\le \|\int |\xi-\zeta| |\widehat{f}(\xi-\zeta)||\widehat{g}(\zeta)|d\zeta\|_2\le \||\widehat{\partial_x f}|\ast|\widehat{g}|\,\|\\
&\le \|\widehat{\partial_x f}\|_1\|g\|_2.
\end{split}
\end{equation*}
\end{proof}

\begin{proposition}\label{prop**}
Given  $\phi\in L^{\infty}(\R)$, with $\partial_x^{\alpha}\phi\in L^2(\R)$ for $\alpha=1,2$, then
for any $\theta\in(0,1)$
\begin{equation}\label{*2}
\|J^{\theta}(\phi f)\|_2\le c_{\theta,\phi}\,\|J^{\theta}f\|_2.
\end{equation}
\end{proposition}

\begin{proof} We just need to write
\begin{equation*}
J^{\theta}(\phi f)=[J^{\theta},\phi] f+\phi\,J^{\theta}f
\end{equation*}
and use the hypotheses and Proposition \ref{prop*} \eqref{*1}.
\end{proof}

 We recall the following  characterization of the $L^p_s(\R^n)\!=\!(1-\Delta)^{-s/2}L^p(\R^n)$ spaces given  in \cite{St1}, (see \cite{ArSm} for the case $p=2$).

 \begin{theorem}[\cite{St1}]
 \label{theorem9b}
Let $b\in (0,1)$ and $\; 2n/(n+2b)< p< \infty$. Then $f\in  L^p_b(\R^n)$ if and only if
\begin{equation}\label{d1}
\;(a)\hskip10pt f\in L^p(\R^n),\hskip3cm\text{\hskip5cm }\\
\end{equation}
\begin{equation}\label{d1b}
\;(b)\hskip10pt \mathcal D^b f(x)=(\int_{\R^n}\frac{|f(x)-f(y)|^2}{|x-y|^{n+2b   }}dy)^{1/2}\in L^p(\R^n),
\text{\hskip2cm}
\end{equation}
with
 \begin{equation}
\label{d1-norm}
\|f\|_{b,p}\equiv  \|(1-\Delta)^{b} f\|_p=\|J^bf\|_p\simeq \|f\|_p+\|D^b    f\|_p\simeq \|f\|_p
+\|\mathcal D^b  f\|_p.
\end{equation}
 \end{theorem}

For the proof of this theorem we refer the reader to  \cite{St1}.
  One sees that from \eqref{d1b} for $p=2$ and $b\in(0,1)$ one has
\begin{equation}\label{pointwise2}
\|\mathcal D^b (fg)\|_2\leq \|f\, \mathcal D^b g\|_2  +\|g\, \mathcal D^b f\|_2,
\end{equation}
and
\begin{equation}\label{pointwise2a}
\|\mathcal D^b f\|_2= c\, \|D^b f\|_2.
\end{equation}
We shall use these estimates throughout in our arguments.

As   an application of Theorem \ref{theorem9b}
we also  have the following estimate:

 \begin{proposition}\label{propositionB}
 Let $b \in (0,1)$. For  any $t>0$
 \begin{equation}
 \label{pointwise-es}
  \mathcal D^b      (e^{-it  |x|^{1+a}x})\leq c(|t|^{b/(2+a)}+|t|^b    |x|^{(1+a)b}   ).
  \end{equation}
 \end{proposition}

For the proof of Proposition \ref{propositionB} we refer to \cite{NaPo}.

Also as consequence of  the estimate \eqref{pointwise2} one has  the following
interpolation inequality.

\begin{lemma}\label{lemma1}
Let $\alpha,\,b>0$. Assume that $ J^{\alpha}f=(1-\Delta)^{\alpha/2}f\in L^2(\mathbb R)$ and \newline
$\ji x\jd^{b}f=(1+|x|^2)^{b/2}f\in L^2(\mathbb R)$. Then for any $\theta \in (0,1)$
\begin{equation}\label{complex}
\|J^{ \theta \alpha}(\ji x\jd^{(1-\theta) b} f)\|_2\leq c\|\ji x\jd^b f\|_2^{1-\theta}\,\|J^{\alpha}f\|_2^{\theta}.
\end{equation}
 Moreover, the inequality \eqref{complex} is still valid with $\ji x\jd_N^{\theta}$ in \eqref{trun} instead of $\langle x\rangle$ with a
constant $c$ independent of $N$.
\end{lemma}
We refer to \cite{FoPo} for the proof of Lemma \ref{lemma1}.

As a further  direct  consequence of Theorem \ref{theorem9b}  we deduce the
following result. It will be useful in several of our arguments.

 \begin{proposition}\label{prop3}
For any $\theta\in (0,1)$ and $\alpha>0$,
\begin{equation}\label{steinderiv}
\mathcal{D}^{\theta}\Big(|\xi|^{\alpha}\chi(\xi)\Big)(\eta)\sim
\begin{cases}
c\,|\eta|^{\alpha-\theta}+c_1,\hskip15pt\alpha\neq \theta,\hskip10pt |\eta| \ll 1,\\
\\
c\,(-\ln|\eta|)^{\frac{1}{2}},\hskip20pt \alpha=\theta,\hskip10pt|\eta| \ll 1,\\
\\
\dfrac{c}{|\eta|^{1/2+\theta}},\hskip78pt |\eta|\gg1,
\end{cases}
\end{equation}
with $\mathcal{D}^{\theta}\Big(|\xi|^{\alpha}\chi(\xi)\Big)(\cdot) $ continuous in $\eta\in\R-\{0\}$.

In particular, one has that
\begin{equation*}
\mathcal{D}^{\theta}\Big(|\xi|^{\alpha}\chi(\xi)\Big)\in L^2(\R)
\text{\hskip10pt if and only if\hskip10pt}\theta<\alpha+1/2.
\end{equation*}
Similar result holds for $\mathcal{D}^{\theta}\Big(|\xi|^{\alpha}\sgn(\xi)\chi(\xi)\Big)(\eta)$.

\end{proposition}

\begin{proof} We restrict ourselves to the case $\alpha\neq \theta$. First we consider the case
$\mathcal{D}^{\theta}\Big(|\xi|^{\alpha}\chi(\xi)\Big)$. It is easy to see
that for $\eta\neq 0$,  $\mathcal{D}^{\theta}\Big(|\xi|^{\alpha}\chi(\xi)\Big)(\eta)$ is continuous in $\epsilon<|\eta|<1/\epsilon$
for any $\,\epsilon>0$.

Let us consider $|\eta|<1/2$. Without loss of generality we assume that $\eta\in (0,1/2)$.  Thus
\begin{equation*}
\begin{split}
\Big(\mathcal{D}^{\theta}\Big(|\xi|^{\alpha}\chi(\xi)&\Big)(\eta)\Big)^2
=\int \dfrac{\Big(|\xi+\eta|^{\alpha}\chi(\xi+\eta)-|\eta|^{\alpha}\chi(\eta)\Big)^2}{|\xi|^{1+2\theta}}\,d\xi\\
&\leq  c\int_{-\eta/2}^{\eta/2}\dfrac{(|\xi+\eta|^{\alpha}-|\eta|^{\alpha})^2}{|\xi|^{1+2\theta}}\,d\xi+
 c\int_{\eta/2}^{2}\dfrac{(|\xi+\eta|^{\alpha}+|\eta|^{\alpha})^2}{|\xi|^{1+2\theta}}\,d\xi\\
&\equiv A_1+A_2.
 \end{split}
 \end{equation*}

 To bound $A_1$  we use that
 \begin{equation*}
 |(\xi+\eta)^{\alpha}-\eta^{\alpha}|\le |\xi|\dfrac{c_{\alpha}}{\eta^{1-{\alpha}}}\quad \forall \xi:\;\;-\eta/2<\xi<\eta/2.
 \end{equation*}
 Thus
 \begin{equation*}
 A_1\le c\int\limits_{-\eta/2}^{\eta/2}\frac{|\xi|^2}{\eta^{2(1-{\alpha})}|\xi|^{1+2\theta}}\,d\xi\le c\eta^{2({\alpha}-\theta)}.
 \end{equation*}

 For $A_2$ we have that
 \begin{equation*}
 |(\xi+\eta)^{\alpha}+\eta^{\alpha}|\le c \xi^{\alpha}\;\;\text{for}\;\;\xi:\;\;\eta/2\le \xi\le 2.
 \end{equation*}
 So
 \begin{equation*}
 A_2\le c\int\limits_{\eta/2}^{2}\dfrac{\xi^{2{\alpha}}}{\xi^{1+2\theta}}\,d\xi\le c\,\eta^{2({\alpha}-\theta)}+c_1,\;\;\;\;\text{if}\;\;\;\alpha\neq \theta,
 \end{equation*}
 where $c_1=0$ if $\theta>\alpha$.

 Let us consider the case $|\eta|>100$. Without loss of generality we assume $\eta>100$. Then
 \begin{equation*}
 \begin{split}
 \Big(\mathcal{D}^{\theta}(|\xi|^{\alpha}\chi(\xi))(\eta)\Big)^2&=\int \dfrac{(|\xi+\eta|^{\alpha}\chi(\xi+\eta))^2}{|\xi|^{1+2\theta}}d\xi\\
&\le c\int\limits_{-2-\eta}^{2-\eta}\frac{d\xi}{|\xi|^{1+2\theta}} \le \dfrac{c}{\eta^{1+2\theta}}.
\end{split}
\end{equation*}

Finally we consider $\mathcal{D}^{\theta}(|\xi|^{\alpha}\sgn(\xi)\chi(\xi))(\eta)$. We notice that the previous
computation for $|\eta|>100$ is similar, so we just need to consider the case $|\eta|<1$. Assuming
without loss of generality that $0<\eta<1$.

Since
\begin{equation*}
\begin{split}
\mathcal{D}^{\theta}(|\xi|^{\alpha}\sgn(\xi)&\chi(\xi))(\xi)=\\
&\int\limits_{-2-\eta}^{2-\eta}\dfrac{\Big(|\xi+\eta|^{\alpha}\sgn(\xi+\eta)\chi(\xi+\eta)
\!-\!|\eta|^{\alpha}\sgn(\eta)\chi(\eta)\Big)^2}{|\xi|^{1+2\theta}}\,d\xi.
\end{split}
\end{equation*}

The bound for $\xi+\eta>0$ is similar to that given before, so we assume $\xi+\eta<0$ ($\xi<-\eta$)
and consider
\begin{equation*}
\int\limits_{-2-\eta}^{-\eta}\frac{(|\xi+\eta|^{\alpha}+|\eta|^{\alpha})^2}{|\xi|^{1+2\theta}}\,d\xi=
\int\limits_{-2\eta}^{-\eta}\dots+\int\limits_{-2-\eta}^{-2\eta}\dots=\tilde{A}_1+\tilde{A}_2.
\end{equation*}

A familiar argument  shows that
\begin{equation*}
\tilde{A}_1\le c\frac{\eta^{2{\alpha}}}{\eta^{1+2\theta}}\,\eta=\eta^{2({\alpha}-\theta)},
\end{equation*}
and
\begin{equation*}
\tilde{A}_2\le \int\limits_{-2-\eta}^{-2\eta}\dfrac{|\xi|^{2{\alpha}}}{|\xi|^{1+2\theta}}\,d\xi=\eta^{2({\alpha}-\theta)}+c,
\end{equation*}
with $c=0$ if $\theta>{\alpha}$.

\end{proof}


\section{Proof of Theorem \ref{lwp-zsr}}

\begin{proposition}\label{prop1} If $u_0\in H^{1+a}(\R)$ and
$|x|^{\theta}\,u_0\in L^2(\R)$, $\theta\in(0,1)$, then the solution
 $u$ of the IVP \eqref{DGBO} satisfies
\begin{equation}
u\in C([0,T]:H^{1+a}(\R)\cap L^2(|x|^{2\theta})\equiv
Z_{1+a,\theta}).
\end{equation}
where $T$ is given by the local theory.
\end{proposition}

\begin{proof}[Proof of Proposition \ref{prop1}]

We use the differential equation and the local theory such that
$u\in C([0,T]:H^{1+a}(\R))$ exists and is the limit of smooth
solutions.

We define for $\theta\in(0,1)$
\begin{equation}\label{trun}
\ji x\jd_{N}^{\theta}=
\begin{cases}
\ji x\jd^{\theta}=(1+x^2)^{\theta/2}, \text{\hskip1cm if}\;\; |x|\le N,\\
(2N)^{\theta}, \text{\hskip85pt if}\;\; |x|\ge 3N,
\end{cases}
\end{equation}
with $\ji x\jd^{\theta}_N$ smooth, even, nondecreasing for $|x|\ge 0$.

We multiply the equation in \eqref{DGBO}
by $\ji x\jd^{2\theta}_N\,u$ and integrate in the
$x$-variable to get
\begin{equation*}\label{e1.2}
\frac12 \frac{d}{dt}\int (\ji x\jd^{\theta}_N\,u)^2\,dx+
\underset{A_1}{\underbrace{\int \ji
x\jd^{\theta}_N\,D^{1+a}\partial_x u\ji x\jd^{\theta}_N\,u\,dx}}+
\underset{A_2}{\underbrace{\int \ji x\jd^{2\theta}_N\, u^2\partial_x
u\,dx}}=0.
\end{equation*}

To estimate $A_2$ we integrate by parts to get
\begin{equation}\label{e1.3}
A_2=\frac13\int \ji
x\jd^{2\theta}_N\,\partial_x(u^3)\,dx=-\frac13\int \partial_x(\ji
x\jd^{2\theta}_N)\,u^3\,dx.
\end{equation}

We shall use that
\begin{equation}\label{e1.4}
\partial_x (\ji x\jd^{2\theta}_N)\le c_{\theta}\ji x\jd^{2\theta-1}_N\le c_{\theta}\,\ji x\jd^{\theta}_N,
\end{equation}
since $\theta\in(0,1)$. Notice that $c_\theta$ is independent of
$N$. Thus
\begin{equation}\label{e1.5}
A_2\le \|\ji x\jd^{\theta}_N\, u\|_2\|u\|_2\|u\|_{\infty}.
\end{equation}

Now we turn to $A_1$. We write
\begin{equation}\label{e1.6}
\ji x\jd^{\theta}_N\, D^{1+a}\partial_x u=D^a(\ji
x\jd^{\theta}_N\,D\partial_x u) -[D^a;\ji
x\jd^{\theta}_N]\,D\partial_x u\equiv B_1+B_2.
\end{equation}

Using Proposition \ref{dmp2}
\begin{equation}\label{e1.7}
\|[D^a;\ji x\jd^{\theta}_N]\,D^{1-a}f\|_2\le
c\,\|J^{\delta}\partial_x \ji x\jd^{\theta}_N\|_q\|f\|_2,
\end{equation}
with $\delta>1/q$. Thus
\begin{equation}\label{e1.8}
\|B_2\|_2=\|[D^a;\ji x\jd^{\theta}_N]\,D\partial_x u\|_2\le
c_{\theta}\,\|D^a\partial_x u\|_2,
\end{equation}
with $c_{\theta}$ independent of $N$ since $\partial_x \ji
x\jd^{\theta}_N$ is bounded independent of $N$. Here we are assuming
that $\theta<1$ such that $J^{\delta}\partial_x \ji
x\jd^{\theta}_N\in L^q$ for appropriate values of $\delta$, $q$ with
$\delta>1/q$. When $\theta=1$, $\|J^{\delta}\partial_x \ji
x\jd^1_N\|_q$ is not bounded uniformly on $N$ by a constant and we
cannot do this.

Also observe that the bound $\|D^a\partial_x u\|_2$ is natural from the fact that
the operator $\Gamma=x-(2+a)tD^{1+a}$ which commutes with $\partial_t+D^{1+a}\partial_x$.

Hence it remains to consider $B_1$ in \eqref{e1.6}. We write
\begin{equation}\label{e1.9}
\begin{split}
B_1=D^a(\ji x\jd^{\theta}_N D\partial_x u)&= D^a\partial_x(\ji x\jd^{\theta}_N\,D u) -D^a((\partial_x \ji x\jd^{\theta}_N)Du)\\
&\equiv C_1+C_2,
\end{split}
\end{equation}
with
\begin{equation}\label{e1.10}
\begin{split}
\|C_2\|_2&=\|D^a((\partial_x \ji x\jd^{\theta}_N) Du)\|_2\\
&\le \|[D^a; \partial_x \ji x\jd^{\theta}_N]Du\|_2+\|\partial_x \ji x\jd^{\theta}_N\,D^{1+a}u\|_2\\
&\underset{(3)}{\le} c\,\|J^{\delta}\partial_x^2\ji
x\jd^{\theta}_N\|_q\|D^{1-a}u\|_2
+c\|D^{1+a}u\|_2\\
&\le c\,\|J^{1+a}u\|_2, \quad c\;\;\text{independent of}\;\; N,
\end{split}
\end{equation}
where in (3) we have used again \eqref{e1.7} (Proposition
\ref{dmp2}).

To estimate $C_1$ in \eqref{e1.9} we write
\begin{equation}\label{e1.11}
\begin{split}
C_1=D^a\partial_x (\ji x\jd^{\theta}_N\,Du)&=D^a\partial_x D( \ji x\jd^{\theta}_N\,u) -D^a\partial_x\,[D;\ji x\jd^{\theta}_N]u\\
&\equiv K_1+K_2.
\end{split}
\end{equation}

Since $D=\mathcal{H}\partial_x$ one has
\begin{equation}\label{e1.12}
\begin{split}
[D;\ji x\jd^{\theta}_N]f&= D(\ji x\jd^{\theta}_N\, f)- \ji x\jd^{\theta}_N\,Df\\
&= \hil\partial_x(\ji x\jd^{\theta}_N\,f)-\ji x\jd^{\theta}_N\,\hil\partial_x f\\
&= \hil((\partial_x\ji x\jd^{\theta}_N)f)-[\hil;\ji
x\jd^{\theta}_N]\partial_x f.
\end{split}
\end{equation}

Therefore
\begin{equation}\label{e1.13}
K_2=-D^a\partial_x \hil ((\partial_x\ji
x\jd^{\theta}_N)\,u)+D^a\partial_x[\hil;\ji
x\jd^{\theta}_N]\partial_x u \equiv Q_1+Q_2.
\end{equation}

To bound $Q_2$ we use the commutator estimate in Proposition
\ref{dmp1}
\begin{equation}\label{e1.14}
\|\partial^j[\hil;a]\partial^m f\|_2\le c\,\|\partial^{j+m}
a\|_{\infty}\|f\|_2,
\end{equation}
and interpolation ($\|D^af\|_2\le \|f\|_2^{1-a}\|Df\|^a_2$) to get
\begin{equation}\label{e1.15}
\begin{split}
\|Q_2\|_2&=\|D^a\partial_x[\hil;\ji x\jd^{\theta}_N]\partial_x u\|_2\\
&\le \|D\partial_x[\hil;\ji x\jd^{\theta}_N]\partial_x u\|_2^a
\|\partial_x[\hil;\ji x\jd^{\theta}_N]\partial_x u\|_2^{1-a}\\
&\le \|\partial_x^2[\hil;\ji x\jd^{\theta}_N]\partial_x u\|_2^a\,
\|\partial_x[\hil;\ji x\jd^{\theta}_N]\partial_x u\|_2^{1-a}\\
&\le c\,\big(\|\partial_x^3\ji
x\jd^{\theta}_N\|_{\infty}+\|\partial_x^2\ji
x\jd^{\theta}_N\|_{\infty}\big)\, \|u\|_2.
\end{split}
\end{equation}

Using previous arguments we also have
\begin{equation}\label{e1.16}
\begin{split}
\|Q_1\|_2&=\|\hil D^a\partial_x ((\partial_x\ji x\jd^{\theta}_N)\,u)\|_2\\
&\le \|D^a ((\partial_x^2\ji x\jd^{\theta}_N)\,u)\|_2+\|D^a((\partial_x\ji x\jd^{\theta}_N)\,\partial_xu)\|_2\\
&\le \big(\|D^a \partial_x^2\ji
x\jd^{\theta}_N\|_2\|u\|_{\infty}+\|\partial_x^2\ji
x\jd^{\theta}_N\|_{\infty}
\|D^au\|_2\big)\\
&\quad +\|[D^a;\partial_x\ji x\jd^{\theta}_N]\partial_x u\|_2
+\|(\partial_x\ji x\jd^{\theta}_N)D^a\partial_x u\|_2\\
&\le c\,\|J^{1+a}u\|_2.
\end{split}
\end{equation}

Finally, we turn to the term $K_1$ in \eqref{e1.11}, Parseval's identity yields
$$
\int D^a\partial_x D( \ji x\jd^{\theta}_N\,u)\,\ji x\jd^{\theta}_N\,
u =\int \partial_x D^{(1+a)/2}(\ji x\jd^{\theta}_N\,u)D^{(1+a)/2}(\ji x\jd^{\theta}_N\,u)\equiv 0.
$$

Since from the local existence theory \cite{Gu2} we know that
\begin{equation}\label{e1.18}
\underset{[0,T]}{\sup}\, \|u(t)\|_{1+a,2}=\underset{[0,T]}{\sup}\,
\|u(t)\|_{H^{1+a}}\le c,
\end{equation}
combining the above estimate and taking limit as $N\to\infty$ we
obtain
\begin{equation}\label{e1.19}
\underset{[0,T]}{\sup}\, \|\ji x\jd^{\theta}u(t)\|_{2}\le
\tilde{c}_{\theta}\quad\text{for all}\;\;\theta\in[0,1),
\end{equation}
which yields the result.

We observe that the argument above shows that if in addition to
$u_0\in H^{1+a}(\R)\cap L^2(|x|^{2\theta})$, $\theta\in (0,1)$,
$u_0\in H^{1+a+\alpha}(\R)$ with $D^{\alpha}u_0\in L^2(|x|^{2\theta})$, for $\alpha>0$,
then $D^{\alpha}u \in C([0,T] : H^{1+a}(\R)\cap L^2(|x|^{2\theta}))$, $\alpha>0$.

\end{proof}

\subsection{Case $s=1+a$, $r=1$} Let $u_0\in H^{1+a}(\R)\cap L^2(|x|^2\,dx)$. We observe that the
persistence result in this case was already proved by Colliander, Kenig and Stafillani \cite{CoKeSt}.
However, by convenience we present a different proof.

First notice that
\begin{equation}\label{e2.B1}
\begin{split}
x\,D^{1+a}\partial_x f&=(x\,D^{1+a}\partial_xf)^{\wedge\vee}=(i\partial_{\xi}\,(|\xi|^{1+a}\,i\xi\widehat{f}\,))^{\vee}\\
&=\big(-(2+a)|\xi|^{1+a}\widehat{f}+|\xi|^{1+a}i\xi i\partial_{\xi}\widehat{f}\,\big)^{\vee}\\
&=-(2+a)\,D^{1+a} f+D^{1+a}\partial_x(xf).
\end{split}
\end{equation}

Hence if $u$ satisfies
\begin{equation}\label{3.25a}
\partial_t u+D^{1+a}\partial_x u+u\partial_x u=0,
\end{equation}
then
\begin{equation}\label{e2.B2}
\partial_t(xu)+D^{1+a}\partial_x (xu)-(2+a)D^{1+a}u+x\,u\partial_x u=0.
\end{equation}

In this case the standard energy estimate argument shows that
\begin{equation}\label{e2.B21}
\frac{d}{dt}\|xu(t)\|_2^2\le c_a\|D^{1+a}u(t)\|_2\|xu(t)\|_2+\|u\|_{\infty}\|u(t)\|_2\|xu(t)\|_2.
\end{equation}
Since
\begin{equation}\label{e2.B22}
\sup_{[0,T]} \|u(t)\|_{1+a,2}\le c(a;T;\|u_0\|_{1+a}),
\end{equation}
one has from \eqref{e2.B21} that
\begin{equation}\label{e2.B23}
\sup_{[0,T]} \|xu(t)\|_{1+a,2}\le c(a;T;\|u_0\|_{1+a,2};\|xu_0\|_2).
\end{equation}

\begin{remark}\label{AAA} By taking derivatives $D^{\alpha}$ in the equation \eqref{3.25a} and repeating
the above argument we have that if in addition to $u_0\in H^{1+a}(\R)\cap L^2(|x|^2\,dx)$ one has
\begin{equation}\label{e2.B24}
D^{\alpha} u_0\in H^{1+a}(\R),\;\; xD^{\alpha}u_0\in L^2(\R), \;\;\text{for some}\;\;\alpha>0,
\end{equation}
then
\begin{equation}\label{e2.B25}
\sup_{[0,T]} \|x D^{\alpha}u(t)\|_2\le c(T;a;\|u_0\|_{1+a+\alpha,2};\|xu_0\|_2;\|xD^{\alpha}u_0\|_2).
\end{equation}
\end{remark}

\subsection{Case $s=2(1+a)$, $r\in(1,2)$} Let $u_0\in H^{2(1+a)}\cap L^2(|x|^{2r}\,dx)$, $r=1+\theta$,
$\theta\in(0,1)$.

Reapplying the method for the weight $\ji x\jd^{2\theta}_N$, $\theta\in (0,1)$, multiplying  the equation
\eqref{e2.B2}  by $\ji x\jd^{2\theta}_N xu$ and integrating the result, one gets
\begin{equation}\label{e2.B2A1}
\begin{split}
&\frac12 \frac{d}{dt} \int \big(\ji x\jd^{\theta}_N xu)^2\,dx
+\int \ji x\jd^{\theta}_ND^{1+a}\partial_x(xu)\ji x\jd^{\theta}_N xu\\
&-(2+a)\int \ji x\jd^{\theta}_N D^{1+a}u \ji x\jd^{\theta}_N xu
+\int \ji x\jd^{\theta}_N xu \partial_x\ji x\jd^{\theta}_N xu=0.
\end{split}
\end{equation}

From the previous analysis and Proposition \ref{prop1},
we only need to handle the last two terms in \eqref{e2.B2A1}.

First we notice that
\begin{equation}\label{e2.B3}
\int \ji x\jd^{\theta}_N\,D^{1+a}u\,\ji x\jd^{\theta}_Nx\,u\,dx\le
\|\ji x\jd^{\theta}_ND^{1+a}u\|_2 \|\ji x\jd^{\theta}_N\,xu\|_2.
\end{equation}
Next
\begin{equation*}
\ji x\jd^{\theta}_N D^{1+a}u= \underset{C_1}{\underbrace{-[D^a;\ji x\jd^{\theta}_N] D^{1-a}D^au}}
+ \underset{C_2}{\underbrace{D^a\ji x\jd^{\theta}_N\,Df}},
\end{equation*}
and by the commutator estimate
\begin{equation*}
\|C_1\|_2\le \|D^a u\|_2 \quad \text{uniformly  in $N$ for}\;\;\theta\in (0,1).
\end{equation*}

On the other hand,
\begin{equation*}
C_2= D^a\ji x\jd^{\theta}_N\,\partial_x\hil u= D^a\partial_x (\ji x\jd^{\theta}_N\,\hil u)
-D^a(\partial_x \ji x\jd^{\theta}_N\, \hil u)\equiv  K_1+K_2,
\end{equation*}
where
\begin{equation*}
\begin{split}
\|K_2\|_2&=\|D^a(\partial_x \ji x\jd^{\theta}_N\, \hil u)\|_2\le \|J^a(\partial_x \ji x\jd^{\theta}_N\, \hil u)\|_2\\
&\le \|[J^a; \partial_x \ji x\jd^{\theta}_N]\hil u\|_2+\|\partial_x \ji x\jd^{\theta}_N\, J^a\hil u\|_2\\
&\le \|u\|_2+\|J^au\|_2\le c\|J^a u\|_2.
\end{split}
\end{equation*}

So we consider $K_1$.
\begin{equation*}
K_1=D^a\partial_x(\ji x\jd^{\theta}_N\,\hil u)
=D^a\hil\partial_x(\ji x\jd^{\theta}_N u)+D^a \partial_x [\ji x\jd^{\theta}_N; \hil]u
= K_{1,1}+K_{1,2}.
\end{equation*}

For $K_{1,2}$ we write
\begin{equation*}
\begin{split}
\|K_{1,2}\|_2&=\|D^a\partial_x [\ji x\jd^{\theta}_N;\hil]u\|_2\\
&\le \|D\partial_x[\ji x\jd^{\theta}_N;\hil]u\|_2^a\|\partial_x[\ji x\jd^{\theta}_N;\hil]u\|^{1-a}_2\\
&\le \|\partial_x^2 \ji x\jd^{\theta}_N\|_{\infty}^a\|u\|_2^a\|\partial_x \ji x\jd^{\theta}_N\|_{\infty}^{1-a}
\|u\|_2^{1-a}\le c_{\theta}\|u\|_2.
\end{split}
\end{equation*}

Finally,
\begin{equation*}
\begin{split}
\|K_{1,1}\|_2&=\|D^{1+a}(\ji x\jd^{\theta}_N u)\|_2\le \|J^{1+a}(\ji x\jd^{\theta}_N u)\|_2\\
&\le \|J^{(1+a)(1+\theta)}u\|^{1/1+\theta}_2\|\ji x\jd^{1+\theta}_N u\|^{\theta/1+\theta}_2\\
&\le \|J^{2(1+a)}u\|+\|\ji x\jd^{\theta}_N \,x u\|_2+\|u\|_2+\|xu\|_2.
\end{split}
\end{equation*}
which completes the estimate for \eqref{e2.B3} if $s\ge 2(1+a)$.

For the nonlinear term coming from \eqref{e2.B2} we have that
\begin{equation*}
\begin{split}
\int \ji x\jd^{\theta}_N \,x\,u \partial_xu \ji x\jd^{\theta}_N \,xu
&\le \|\ji x\jd^{\theta}xu\|^2_2\|\partial_xu\|_{\infty}\le  \|\ji x\jd^{\theta}xu\|^2_2\|u\|_{2(1+a),2}.
\end{split}
\end{equation*}


This proves that if $u_0\in H^{2(1+a)}(\R)$ and $|x|^{1+\theta}\,u\in L^2(\R)$, $\theta\in (0,1)$,
then persistence holds in
\begin{equation*}
H^{2(1+a)}\cap L^2(|x|^{2(1+\theta)}\,dx).
\end{equation*}

\begin{remark}\label{AAA2} The above argument also shows that if in addition to
$u_0\in H^{2(1+a)}(\R)\cap L^2(|x|^{2(1+\theta)}\,dx)$ one has
$D^{\alpha} u_0\in H^{2(1+a)}(\R)\cap L^2(|x|^{2(1+\theta)}\,dx)$, $\alpha>0$,
then $D^{\alpha} u\in C([0,T]: H^{2(1+a)}(\R)\cap L^2(|x|^{2(1+\theta)}\,dx))$.
\end{remark}

\subsection{Case $s=2(1+a)$, $r=2$} We observe that an argument similar to that in \cite{CoKeSt} also gives the
persistence of the solution to the IVP \eqref{DGBO} in
\begin{equation}\label{B1}
u_0\in H^{2(1+a)}(\R)\cap L^2(|x|^{4}\,dx).
\end{equation}

In fact, using for
$$
\aligned
&x^2D^{1+a}\partial_x f
= (x^2 D^{1+a}\partial_x f)^{{\wedge}{\vee}}
=(-\partial_{\xi}^2(|\xi|^{1+a}i\xi \widehat{f}))^{\vee}\\
&\;\;=(-(2+a)(1+a)|\xi|^a\,i\sgn (\xi)\,\widehat{f}-2(2+a)|\xi|^{1+a}\widehat{x\,f}
+|\xi|^{1+a}i\xi \widehat{(x^2\,f)})^{\vee}\\
&\;\;=(2+a)(1+a) D^a\hil f - 2(2+a)D^{1+a}(xf)+D^{1+a}\partial_x(x^2f).
\endaligned
$$
We get the equation for $x^2u$
$$
\partial_t(x^2u)+D^{1+a}\partial_x(x^2 u)-2(2+a)D^{1+a}(xu)-(1+a)(2+a) D^a\hil u-x^2u\partial_x u=0,
$$
for which a familiar argument also shows that
\begin{equation}\label{B2}
\sup_{[0,T]} \|x^2 u(t)\|_2 \le c (T;a; \|x^2 u_0\|_2;\|u_0\|_{2(1+a),2}).
\end{equation}

As before we notice that if in addition to \eqref{B1} one has that
\begin{equation}\label{B3}
D^{\alpha}u_0\in H^{2(1+a)}(\R)\cap L^2(|x|^{4}\,dx)\equiv Z_{2(1+a),2},\;\;\alpha>0,
\end{equation}
then
\begin{equation}\label{B4}
\sup_{[0,T]} \|x^2 D^{\alpha}u(t)\|_2\le c(T;a;\|x^2u_0\|_2;\|x^2D^{\alpha}u_0\|_2;\|u_0\|_{2(1+a)+\alpha,2}).
\end{equation}

\subsection{Case $s=s_r\equiv[(r+1)^-](1+a)$, $\,r\in (2,5/2+a)$} We observe that the equation for $xu$
\begin{equation}\label{B6}
\partial_t(xu)+D^{1+a}\partial_x(xu)-(2+a)D^{1+a}u+(xu)\partial_xu=0,
\end{equation}
and the previous argument for $|x|^l$, with $l\in(0,2)$ will provide the result if the
contribution for the extra term in \eqref{B6} $c_a\,D^{1+a}u$ can be handled, i.e. if
for $l=1+\theta$
\begin{equation}\label{B7}
\sup_{[0,T]} \||x|^{1+\theta}D^{1+a}u(t)\|_2\le M.
\end{equation}

We claim that if $\theta\in (0,a+1/2)$, $u_0\in H^{s_r}(\R)\cap L^2(|x|^{2+\theta})$,
$s=r[(r+1)^-](1+a)$, we obtain \eqref{B7} and hence the desired result. From Remark \ref{AAA} it will suffice to have
\begin{equation}\label{B8}
\||x|^{1+\theta} D^{1+a} u_0\|_2\le c(\||x|^{2+\theta}u_0\|_2;\|D^{(1+a)(1+\theta)}u_0\|_2).
\end{equation}
with $c$ independent of $N$.

\begin{proof}[Proof of \eqref{B8}] Using the identity
\begin{equation}\label{B9}
x\,D^{1+a}f=D^{1+a}(xf)+(1+a)D^a\hil f,
\end{equation}
we have to control the $L^2$-norms of the terms
\begin{equation}\label{B10}
K_1=|x|^{\theta} D^a\hil u_0\text{\hskip10pt and\hskip10pt} K_2=|x|^{\theta} D^{1+a}(xu_0).
\end{equation}

We can estimate $K_1$ as
\begin{equation}\label{B11}
\begin{split}
\|K_1\|_2\le& \|D^{\theta}_{\xi}(|\xi|^a\sgn(\xi)\chi(\xi)\widehat{u}_0(\xi))\|_2\\
&+\|D^{\theta}_{\xi}(|\xi|^a\sgn(\xi)(1-\chi(\xi))\widehat{u}_0(\xi))\|_2
=K_{1,1}+K_{1,2}.
\end{split}
\end{equation}
where
\begin{equation*}
\begin{split}
K_{1,1} &\le \|D^{\theta}_{\xi}(|\xi|^a\sgn(\xi)\chi(\xi))\widehat{u}_0(0))\|_2
+\|D^{\theta}_{\xi}(\underset{L(\xi)}{\underbrace{|\xi|^a\sgn(\xi)\chi(\xi)(\widehat{u}_0(\xi)-\widehat{u}_0(0)}}))\|_2\\
&=\tilde{K}_{1,1}+\|D^{\theta}_{\xi} L\|_2.
\end{split}
\end{equation*}

Next we have that
\begin{equation*}
\|D^{\theta}_{\xi} L\|_2\le \|L(\xi)\|^{1-\theta}_2\|\partial_{\xi}L(\xi)\|^{\theta}_2.
\end{equation*}

Then
\begin{equation*}
\begin{split}
\|\partial_{\xi}L\|&\le c\big(\|\partial_{\xi}\widehat{u}_0\|_{\infty}+\|D^{\theta}_{\xi}\widehat{u}_0\|_2\big)\\
&\le c\big(\|\widehat{xu_0}\|_{\infty}+\|\widehat{|x|^{\theta}u_0}\|_2\big)\\
&\le c\big(\|xu_0\|_{1}+\|\ji x\jd^{\theta}u_0\|_2\big)\le c\|\ji x\jd^{\frac{3}{2}^+} u_0\|_2,
\end{split}
\end{equation*}
and
\begin{equation*}
\tilde{K}_{1,1}< \infty.
\end{equation*}
(by Proposition \ref{prop3}).
\end{proof}

Consider now $\|K_2\|_2$, we introduce the cutoff function $\chi$ to obtain

\begin{equation*}
\begin{split}
\|K_2\|_2&=\||x|^{\theta}D^{1+a}(xu_0)\|_2\\
&=\|D_{\xi}^{\theta}(|\xi|^{1+a}(\widehat{xu_0}))\|_2\\
&\le \|D_{\xi}^{\theta}(|\xi|^{1+a}\chi(\xi)(\widehat{xu_0}))\|_2+\|D_{\xi}^{\theta}(|\xi|^{1+a}(1-\chi(\xi))(\widehat{xu_0}))\|_2\\
&\le K_{2,1}+K_{2,2}.
\end{split}
\end{equation*}

Using Stein's derivative, Leibniz rule \eqref{pointwise2} and Proposition \ref{prop3}, we estimate $K_{2,1}$ as

\begin{equation*}
\begin{split}
K_{2,1} &=\|D_{\xi}^{\theta}(|\xi|^{1+a}\chi(\xi)(\widehat{xu_0}))\|_2\\
&\le c \|\mathcal D_{\xi}^{\theta}(|\xi|^{1+a}\chi(\xi))\|_{\infty} \|\widehat{xu_0}\|_2+\||\xi|^{1+a}\chi(\xi)\|_{\infty}\|\mathcal D_{\xi}^{\theta}(\widehat{xu_0})\|_2\\
&\le c_a\|xu_0\|_2+c_a\|D_{\xi}^{\theta}(\widehat{xu_0})\|_2\\
&\le c_a \|\ji x\jd^{1+\theta}u_0\|_2.
\end{split}
\end{equation*}

On the other hand, we notice that $\varphi_1(\xi)=\frac{|\xi|^{1+a}(1-\chi(\xi))}{\ji \xi\jd^{1+a}}$ and $\varphi_2(\xi)=\frac{\xi}{\ji \xi\jd}$ satisfy the hypothesis in Proposition \ref{prop**}, and hence it follows that

\begin{equation*}
\begin{split}
K_{2,2}&=\|D_{\xi}^{\theta}(|\xi|^{1+a}(1-\chi(\xi))(\widehat{xu_0}))\|_2\\
&\le \|J^{\theta}_{\xi}\Big(\frac{|\xi|^{1+a}(1-\chi(\xi))}{\ji \xi\jd^{1+a}}\ji\xi\jd^{1+a}\widehat{xu_0}\Big)\|_2\\
&\le c\|J^{\theta}_{\xi}(\ji\xi\jd^{1+a}\partial_{\xi}\widehat{u}_0)\|_2\\
&\le c\|J^{\theta}_{\xi}\partial_{\xi}(\ji\xi\jd^{1+a}\widehat{u}_0)\|_2+\|J^{\theta}_{\xi}(\partial_{\xi}(\ji\xi\jd^{1+a})\widehat{u}_0)\|_2\\
&\le c\,\|J^{1+\theta}_{\xi}(\ji\xi\jd^{1+a}\widehat{u}_0)\|_2+\|J^{\theta}_{\xi}(\frac{\xi}{\ji\xi\jd}\ji\xi\jd^{a}\widehat{u}_0)\|_2\\
&\le c\,\|\ji x\jd^{1+\theta}J^{1+a}u_0\|_2+\|\ji x\jd^{\theta}J^a u_0\|_2\\
&\le c_a\|\ji x\jd^{2+\theta}u_0\|^{1-1/(2+\theta)}_2
\|J^{(2+\theta)(1+a)}u_0\|_2^{1/(2+\theta)}\\
&\;\;\;\;+ c_a\|\ji x\jd^{2+\theta}u_0\|^{\theta/(2+\theta)}_2
\|J^{(2+\theta)\frac{a}{2}}u_0\|_2^{2/(2+\theta)},
\end{split}
\end{equation*}
where in the last inequality we have applied  complex interpolation ($u_0\in H^{s_r}(\R)\cap L^2(|x|^{2r}\,dx)$ with $r=2+\theta\in(2,5/2+a)$).





\subsection{Persistence property in $\dot{Z}_{s_r,r}$ with $s_r=[(r+1)^-](1+a)$ and $r\in[5/2+a,7/2+a)$} To simplify the
exposition we assume $r\in[3,7/2+a)$.

We have established persistence in
\begin{equation}\label{D2}
{Z}_{s_r,r} \hskip10pt \text{with}\hskip10pt r\in[2,5/2+a),
\end{equation}
for the equation (we are assuming $a\in (0,1/2)$ for simplicity)
\begin{equation}\label{D3}
\partial_t u+D^{1+a}\partial_x u+u\partial_x u=0,
\end{equation}
and that $xu$ satisfies the equation
\begin{equation}\label{D4}
\partial_t(xu)+D^{1+a}\partial_x(xu)-(2+a)D^{1+a}u+xu\partial_x u=0.
\end{equation}
Thus if we prove that $u_0\in \dot{Z}_{s_r,r}$ with $r\in[5/2+a,7/2+a)$, then the solution $u$ satisfies
\begin{equation}\label{D5}
|x|^{\alpha} D^{1+a}u\in L^2(\R)\hskip20pt\text{for}\hskip20pt\alpha=r-1\in[2,5/2+a),
\end{equation}
the argument for proving the result in $Z_{s_r,r}$ as in \eqref{D2} will provide the result.

Since $\widehat{u}(0,t)=\widehat{u}_0(0)=\int u_0(x)\,dx$ is preserved by the solution flow, it will suffice to show that
if $u_0\in \dot{Z}_{r(1+a),r}$, $r\in[5/2+a, 7/2+a)$, then $|x|^{\alpha}D^{1+a}u_0\in L^2(\R)$ for $\alpha=r-1$.

Since $\alpha\in[2,5/2+a)$ with $a\in(0,1/2)$ write $\alpha=2+\theta$, $\theta\in(0,1)$, and use that
\begin{equation}\label{D6}
x^2D^{1+a}f=-(1+a)aD^{a-1}f-2(1+a)D^a\hil(xf)+D^{1+a}(x^2f).
\end{equation}

Thus
\begin{equation}\label{D7}
\begin{split}
|x|^{2+\theta} D^{1+a}u_0=&\,-(1+a)a|x|^{\theta}D^{a-1}u_0-2(1+a)D^a\hil (xu_0)\\
&\;+|x|^{\theta}D^{1+a}(x^2u_0)\\
\equiv &\;G_1+G_2+G_3.
\end{split}
\end{equation}

First we write $G_1$ using that
\begin{equation*}
\begin{split}
\||x|^{\theta} D^{a-1}u_0\|_2&=\|D^{\theta}_{\xi}(|\xi|^{a-1}\widehat{u}_0)\|_2\\
&\le \|D^{\theta}_{\xi}(|\xi|^{a-1}\chi(\xi)\widehat{u}_0)\|_2+\|D^{\theta}_{\xi}(|\xi|^{a-1}(1-\chi(\xi))\widehat{u}_0)\|_2.
\end{split}
\end{equation*}
Now using that $\widehat{u}_0(0)=0$, the Taylor expansion allows to write
\begin{equation}
\widehat{u}_0(\xi)=\xi\partial_{\xi}\widehat{u}_0(0)+\int_0^{\xi}(\xi-\zeta)\partial_{\xi}^2\widehat{u_0}(\zeta)d\zeta.
\end{equation}

So
\begin{equation*}
\begin{split}
|\xi|^{a-1}\chi(\xi)\widehat{u}_0(\xi)&=|\xi|^a \sgn(\xi)\chi(\xi)\partial_{\xi}\widehat{u}_0(\xi)
+\chi(\xi)|\xi|^{a-1}\,\int_0^{\xi}(\xi-\zeta)\partial_{\xi}^2\widehat{u}(\zeta)d\zeta\\
&\equiv |\xi|^a \sgn(\xi)\chi(\xi)\partial_{\xi}\widehat{u}_0(\xi)+Q_1.
\end{split}
\end{equation*}
and by Proposition \ref{prop3}
\begin{equation*}
\begin{split}
D^{\theta}_{\xi}\big(|\xi|^a \sgn(\xi)\chi(\xi)\partial_{\xi}\widehat{u}_0(\xi)\big)\in L^2(\R)
\hskip10pt &\text{if and only if}\hskip10pt \theta<a+1/2\\
&\text{if and only if}\hskip10pt \alpha=2+\theta<5/2+a,
\end{split}
\end{equation*}
which holds from our hypotheses, and
\begin{equation*}
\|D^{\theta}_{\xi}Q_1\|_2\le \|Q_1\|_2^{1-\theta}\|\partial_{\xi}Q_1\|_2^{\theta}.
\end{equation*}

\begin{equation*}
\begin{split}
\|Q_1\|_2&\le \|\chi(\xi)|\xi|^{a+1}\|\partial_{\xi}^2\widehat{u}_0\|_{L^{\infty}_{\{|\xi|<1\}}}\|_2
\le \|\widehat{x^2u_0}\|_{\infty}\le c\|x^2u_0\|_{1}\\
&\le c\|\ji x\jd^{\frac{5}{2}^+}u_0\|_2,
\end{split}
\end{equation*}
and
\begin{equation*}
\begin{split}
\|\partial_{\xi}Q_1\|_2&\le \|\chi'(\xi)|\xi|^{a-1}\int_0^{\xi}(\xi-\zeta)\partial_{\xi}^2\widehat{u}_0(\zeta)\,d\zeta\|_2\\
&\;\;\;\;+\|\chi(\xi)(a-1)|\xi|^{a-2}\int_0^{\xi}(\xi-\zeta)\partial_{\xi}^2\widehat{u}_0(\zeta)\,d\zeta\|_2\\
&\;\;\;\;+\|\chi(\xi)|\xi|^{a-1}\int_0^{\xi}\partial_{\xi}^2\widehat{u}_0(\zeta)\,d\zeta\|_2\\
&\le \|\partial_{\xi}^2\widehat{u}_0\|_{\infty}+c_a\|\chi(\xi)|\xi|^a\|\partial_{\xi}^2\widehat{u}_0\|_{\infty}\|_2\\
&\le  c_a \|\partial_{\xi}^2\widehat{u}_0\|_{\infty}\le c\|\ji x\jd^{\frac{5}{2}^+}u_0\|_2.
\end{split}
\end{equation*}
This provides the bound of $G_1$ in \eqref{D7}.

For $G_2$ we write
\begin{equation*}
\begin{split}
\||x|^{\theta}&D^a\hil(xu_0)\|_2=\|D^{\theta}_{\xi}(|\xi|^a\sgn(\xi)(\widehat{xu_0}))\|_2\\
&\le \|D^{\theta}_{\xi}(|\xi|^a\sgn(\xi)\chi(\xi)(\widehat{xu_0}))\|_2
+\|D^{\theta}_{\xi}(|\xi|^a\sgn(\xi)(1-\chi(\xi))(\widehat{xu_0}))\|_2\\
&\le \|\widehat{xu_0}\|_{\infty}+\|\widehat{xu_0}\|_2+\|J^{\theta}_{\xi}(\ji\xi\jd^a\widehat{xu_0})\|_2\\
&\le \|xu_0\|_{1}+\|xu_0\|_2+\|J^{1+\theta}(\ji \xi\jd^a\widehat{u}_0)\|_2
+\|J^{\theta}(\ji \xi\jd^{a-1}\widehat{u}_0)\|_2\\
&\le \|\ji x\jd^{\frac{5}{2}^+}u_0\|_2+\|\ji x\jd^{1+\theta}J^{a}u_0\|_2+\|\ji x\jd^{\theta}u_0\|_2.
\end{split}
\end{equation*}
To complete the estimate we use that
\begin{equation*}
\|\ji x\jd^{1+\theta}J^{\theta}u_0\|_2\le \|\ji x\jd^3u_0\|_2^{(1+\theta)/3}\|J^mu_0\|_2^{(2-\theta)/3},
\end{equation*}
holds if $\;m\Big(\dfrac{2-\theta}{3}\Big)=a<r(1+a)\Big(\dfrac{2-\theta}{3}\Big)$. Since $r\ge3$ we just need $(1+a)(2-\theta)>a$, which holds since $\theta\in(0,1)$.
This finishes the bound for $G_2$.

Finally, we consider $G_3$ in \eqref{D7}.
\begin{equation*}
\begin{split}
\|&|x|^{\theta}D^{1+a}(x^2u_0)\|_2=\|D^{\theta}_{\xi}(|\xi|^{1+a}\sgn(\xi)\partial_{\xi}^2\widehat{u}_0)\|_2\\
&\le \|D^{\theta}_{\xi}(|\xi|^{1+a}\sgn(\xi)\chi(\xi)\partial_{\xi}^2\widehat{u}_0)\|_2
+\|D^{\theta}_{\xi}(|\xi|^{1+a}\sgn(\xi)(1-\chi(\xi))\partial_{\xi}^2\widehat{u}_0)\|_2\\
&= A_1+A_2.
\end{split}
\end{equation*}

The Leibniz rule and Stein derivatives give
\begin{equation*}
\begin{split}
A_1&\le \|D^{\theta}_{\xi}(|\xi|^{1+a}\sgn(\xi)\chi(\xi))\|_{\infty}\|\partial_{\xi}^2\widehat{u}_0\|_2
+\||\xi|^{1+a}\sgn(\xi)\chi(\xi)\|_{\infty}\|D^{\theta}_{\xi}\partial_{\xi}^2\widehat{u}_0\|_2\\
& \le \|\ji x\jd^{2+\theta}u_0\|_2,
\end{split}
\end{equation*}
and
\begin{equation*}
\begin{split}
A_2&\le \|D^{\theta}_{\xi}\big(\partial^2_{\xi}(|\xi|^{1+a}\sgn(\xi)(1-\chi(\xi))\widehat{u}_0)\big)\|_2\\
&\;\;\;+ \|D^{\theta}_{\xi}\big(\partial_{\xi}(|\xi|^{1+a}\sgn(\xi)(1-\chi(\xi))\partial_{\xi}\widehat{u}_0\big)\|_2\\
&\;\;\;+ \|D^{\theta}_{\xi}\big(\partial^2_{\xi}(|\xi|^{1+a}\sgn(\xi)(1-\chi(\xi))\widehat{u}_0\big)\|_2\\
&\equiv K_1+K_2+K_3.
\end{split}
\end{equation*}

Notice that
\begin{equation*}
\partial_{\xi}(|\xi|^{1+a}\sgn(\xi)(1-\chi(\xi)))\!=\!(1+a)|\xi|^a\sgn(\xi)(1-\chi(\xi))+|\xi|^{1+a}\sgn(\xi)\chi'(\xi),
\end{equation*}
and
\begin{equation*}
\begin{split}
\partial_{\xi}^2(|\xi|^{1+a}\sgn(\xi)(1-\chi(\xi)))=&\,c\,|\xi|^{a-1}\sgn(\xi)(1-\chi(\xi))\\
&+c|\xi|^a\sgn(\xi)\chi'(\xi)+|\xi|^{1+a}\sgn(\xi)\chi''(\xi).
\end{split}
\end{equation*}

So to bound $K_3$ we just need to consider
\begin{equation*}
\begin{split}
\|D^{\theta}_{\xi}(|\xi|^{a-1}\sgn(\xi)(1-\chi(\xi))\widehat{u}_0\|_2
&\le \|D^{\theta}_{\xi}(|\xi|^{a-1}\sgn(\xi)(1-\chi(\xi)))\|_{\infty}\|\widehat{u}_0\|_2\\
&\;\;\;+\| |\xi|^{a-1}\sgn(\xi)(1-\chi(\xi))\|_{\infty}\|D^{\theta}_{\xi}\widehat{u}_0\|_2\\
&\le \|u_0\|_2+\| |x|^{\theta}u_0\|_2.
\end{split}
\end{equation*}

To bound $K_2$ we just need to consider
\begin{equation*}
\begin{split}
\|D^{\theta}_{\xi}(|\xi|^a\sgn(\xi)(1-\chi(\xi))\partial_{\xi}\widehat{u}_0\|_2
&\le \|D^{\theta}_{\xi}\Big(\dfrac{|\xi|^a\sgn(\xi)(1-\chi(\xi))}{\ji \xi\jd^a}\ji \xi\jd^a\partial_{\xi}\widehat{u}_0\Big)\|_2\\
&\le  \|J^{\theta}_{\xi}\Big(\dfrac{|\xi|^a\sgn(\xi)(1-\chi(\xi))}{\ji \xi\jd^a}\ji \xi\jd^a\partial_{\xi}\widehat{u}_0\Big)\|_2
\end{split}
\end{equation*}

Now notice that
\begin{equation*}
\Phi(\xi)=\dfrac{|\xi|^a\sgn(\xi)(1-\chi(\xi))}{\ji \xi\jd^a}\in L^{\infty},\;\;
\partial_{\xi}\Phi \cong \frac{1}{\ji x\jd},\;\partial_{\xi}^2\Phi \cong \frac{1}{\ji
x\jd^2} \in L^2.
\end{equation*}

Thus Proposition \ref{prop**}  and interpolation (Lemma \ref{lemma1})  yield
\begin{equation*}
\begin{split}
 \|J^{\theta}_{\xi}\Big(\dfrac{|\xi|^a\sgn(\xi)(1-\chi(\xi))}{\ji \xi\jd^a}&\ji \xi\jd^a\partial_{\xi}\widehat{u}_0\Big)\|_2
 \le \|J^{\theta}_{\xi}\big(\ji \xi\jd^a\partial_{\xi}\widehat{u}_0\big)\|_2\\
 &\le \|\partial_{\xi}J^{\theta}_{\xi}\big(\ji \xi\jd^a\widehat{u}_0\big)\|_2+ \|J^{\theta}_{\xi}\big(\ji \xi\jd^{a-1}\widehat{u}_0\big)\|_2\\
 &\le \|J^{1+\theta}_{\xi}\big(\ji \xi\jd^a\widehat{u}_0\big)\|_2+\|J^{\theta}_{\xi}\widehat{u}_0\|_2\\
 &\le \|\ji x\jd^{1+\theta}J^au_0\|_2+\| |x|^{\theta}u_0\|_2.
 \end{split}
 \end{equation*}

 Finally to bound $K_3$ we just need to consider
 \begin{equation*}
 \begin{split}
\|D^{\theta}_{\xi}\big(|\xi|^{a-1}\sgn(\xi)(1-\chi(\xi))\widehat{u}_0\big)\|_2&\le \|J^{\theta}_{\xi}(|\xi|^{a-1}\sgn(\xi)(1-\chi(\xi))\widehat{u}_0\big)\|_2\\
&\le   \|J^{\theta}_{\xi}\widehat{u}_0\|_2\le \|\ji x\jd^{\theta}u_0\|_2,
\end{split}
\end{equation*}
where we have used Proposition \ref{prop**}.


\begin{section}{Proof of Theorem \ref{theorem7}}

Without loss of generality we assume that $t_1=0<t_2$.

We consider the case $0<a<1/2$ and hence $\frac{5}{2}+a=2+\alpha<3.$

 From the hypothesis we have that
 $u\in C([-T,T]:Z_{(1+a)(\frac{5}{2}+a)+\frac{1-a}{2},\frac{5}{2}+a-\epsilon}),$
 for some $0<\epsilon<\!\!1.$ Therefore
 $$
 u\partial_x u\in C([-T,T]:Z_{(1+a)(\frac{5}{2}+a)-\frac{1+a}{2},4+2a-2\epsilon}),
 $$
 and
 $$
 u\partial_x u\in L^1([-T,T]:H^{s_0}(\R)),\;\;\;\;\text{with}\;\;\;\;s_0\in (0,(1+a)(\frac{5}{2}+a)).
 $$

The solution to the IVP \eqref{DGBO} can be represented by Duhamel's formula
\begin{equation}
  u(t)=W_a(t)u_0-\int_0^t W_a(t-t')u(t')\partial_x u(t') dt',
\label{duh}
\end{equation}
or equivalently in Fourier space as
\begin{equation}
\widehat{u}(\xi,t)=e^{-it|\xi|^{1+a}\xi}\widehat{u}_0(\xi)-\frac{i}{2}\int_0^t e^{-i(t-t')|\xi|^{1+a}\xi}\xi\widehat{u^2}(\xi, t') dt'.
\label{duh1}
\end{equation}

With the notation introduced in \eqref{notation2}, we have
\begin{equation}
\begin{aligned}
\partial_{\xi}^2\widehat{u}(\xi,t)&=F_2(t,\xi,\widehat{u}_0)-\frac{i}{2}\int_0^tF_2(t-t',\xi,\xi \widehat{u^2}(\xi,t'))\\
&=\sum_{1}^{4}B_j(t,\xi,\widehat{u}_0)-\frac{i}{2}\int_0^t\sum_{1}^{4}B_j(t-t',\xi,\xi \widehat{u^2}\,) dt'.
\end{aligned}
\end{equation}

We notice that for any $t\in\R$, and any $j=2,3,4$, we have

\begin{claim1}
Let  $\alpha=\frac{1}{2}+a\in (\frac{1}{2},1)$ and $j=2,3,4$. Then
\begin{equation}
B_j(t,\xi,\widehat{u}_0)-\frac{i}{2}\int_0^tB_j(t-t',\xi,\xi \widehat{u^2}) dt'\in H^{\alpha}(\R),
\label{claim1}
\end{equation}
for all $t\in\R.$
\end{claim1}

 If we assume \eqref{claim1}, it follows that

 \begin{equation}
 \begin{aligned}
 \partial_{\xi}^2\widehat{u}(\xi,t)&\in H^{\alpha}(\R) \,\,\,\text {if and only if}\\
B_1(t,\xi,\widehat{u}_0)-&\frac{i}{2}\int_0^tB_1(t-t',\xi,\xi \widehat{u^2}\,) \,dt'\in H^{\alpha}(\R).
 \end{aligned}
\label{cond1}
 \end{equation}
 We split now $B_1$ as
 \begin{equation}
 \begin{aligned}
 B_1(t,\xi,\widehat{u}_0)&=c_1 t|\xi|^a\text{sgn}(\xi)e^{-it|\xi|^{1+a}\xi}\widehat{u}_0(\xi)\\
 &=c_1 t|\xi|^a\text{sgn}(\xi)e^{-it|\xi|^{1+a}\xi}\widehat{u}_0(\xi)(\chi(\xi)+(1-\chi(\xi)))\\
 &=B_{1,1}+B_{1,2},
 \end{aligned}
 \end{equation}
where $c_1=-i(2+a)(1+a).$ From the hypothesis, it is easy to check that for any $t\in \R-\{0\},$ $B_{1,2}\in H^{1}(\R)$.

 Next, we consider $B_{1,1}$
 $$
  \aligned
 B_{1,1}&=c_1 t|\xi|^a\text{sgn}(\xi)\chi (\xi) (e^{-it|\xi|^{1+a}\xi}-1)\widehat{u}_0(\xi)
 +c_1 t|\xi|^a\text{sgn}(\xi)\chi(\xi)\widehat{u}_0(\xi)\\
 &=\tilde{B}_{1,1}^1+\tilde{B}_{1,1}^2.
 \endaligned
$$
 Once again, we can easily check that for any $t\in \R-\{0\}$, $\tilde{B}_{1,1}^1\in H^{1}(\R). $

 We rewrite  $\tilde{B}_{1,1}$ as:
 \begin{equation}
 \begin{aligned}
 \tilde{B}_{1,1}^2=c_1 t|\xi|^a\text{sgn}(\xi)\chi (\xi)(\widehat{u}_0(\xi)-\widehat{u}_0(0))+c_1 t|\xi|^a\text{sgn}(\xi)\chi (\xi)\widehat{u}_0(0)
  \end{aligned}
 \end{equation}
 and notice that for any $t\in \R-\{0\},$ $c_1 t|\xi|^a\text{sgn}(\xi)\chi(\xi)(\widehat{u}_0(\xi)-\widehat{u}_0(0))\in H^{1}(\R).$

 Now, we apply the  above argument for the inhomogeneous term
 $$\int_0^tB_1(t-t',\xi,\xi \widehat{u^2}) dt',$$
 to conclude that
 \begin{equation}
 \begin{aligned}
\big(&B_1(t,\xi,\widehat{u}_0)-\frac{i}{2}\int_0^tB_1(t-t',\xi,\xi \widehat{u^2}) dt'\big)
-c_1\big( t|\xi|^a\text{sgn}(\xi)\chi(\xi)\widehat{u}_0(0)\\
&-\frac{i}{2}\int_0^t(t-t')|\xi|^a\text{sgn}(\xi)\chi(\xi)\xi\,\widehat{u^2}(0,t')\big)
\in H^{1}(\R)
\label{cond2}
\end{aligned}
 \end{equation}
 and therefore from \eqref{cond1} and \eqref{cond2} we have that for any $t\in \R-\{0\}$
\begin{equation*}\label{cond3}
 \begin{aligned}
 &\partial_{\xi}^2\widehat{u}(\xi,t)\in H^{\alpha}(\R) \,\,\,\,\,\text {if and only if}\\
 t|\xi|^a\text{sgn}(\xi)\chi(\xi)\widehat{u}_0(0)&-\frac{i}{2}\int_0^t(t-t')|\xi|^a\text{sgn}(\xi)\chi(\xi)\xi\,\widehat{u^2}(0,t')
\in H^{\alpha}(\R).
\end{aligned}
 \end{equation*}

Finally, we observe that  for any $t\in \R-\{0\}$
$$
\frac{i}{2}\int_0^t(t-t')|\xi|^a\text{sgn}(\xi)\chi(\xi)\xi\,\widehat{u^2}(0,t')
\in H^{1}(\R),
$$
and hence
 \begin{equation}\label{cond4}
 \partial_{\xi}^2\widehat{u}(\xi,t)\in H^{\alpha}(\R) \;\;\;\text {if and only if}\;\;\;
t|\xi|^a\text{sgn}(\xi)\chi(\xi)\widehat{u}_0(0)\in H^{\alpha}(\R)
\end{equation}
and since $\alpha=\frac{1}{2}+a$, it follows from Proposition \ref{prop3}, that
\eqref{cond4} holds at $t=t_2$ if and only if
$\widehat{u}_0(0)=0.$

In order to complete the proof we go back to Claim 1.
\vskip.5mm

{\it Proof of Claim 1}: We will only give the details in the case
$j=2$. We have
\begin{equation}
B_2(t,\xi,\widehat{u}_0)=
c_2e^{-it|\xi|^{1+a}\xi}t^2|\xi|^{2(a+1)}\widehat{u}_0(\xi)
\end{equation}
and therefore

\begin{equation}
\|B_2(t,\cdot,\widehat{u}_0)\|_2\leq c_t\|D^{2(a+1)}u_0\|_2\leq c_t\|u_0\|_{2(a+1),2}
\end{equation}
and from Proposition \ref{propositionB}, Proposition \ref{prop**} and Lemma \ref{lemma1}
\begin{equation}
\begin{aligned}
\|&D^{\alpha}B_2(t,\cdot,\widehat{u}_0)\|_2\\
&\leq  c_t(\|u_0\|_{2(a+1),2}+\||\xi|^{(2+\alpha)(2+a)}\widehat{u}_0\|_2
+ \|\mathcal D_\xi ^{\alpha}(|\xi|^{2(a+1)}\widehat{u}_0)\|_2)\\
&\leq c_t(\|u_0\|_{(2+\alpha)(a+1),2}+\|\mathcal D_\xi ^{\alpha}(\frac{|\xi|^{2(a+1)}}{\langle \xi\rangle ^{2(1+a)}}{\langle \xi\rangle}^{2(1+a)} \widehat{u} _0)\|_2)\\
&\leq c_t(\|u_0\|_{(5/2+a)(a+1),2}+\|J_\xi ^\alpha (\langle \xi\rangle ^{2(1+a)}\widehat{u} _0)\|_2)\\
&\leq c_t(\|u_0\|_{(5/2+a)(a+1),2}+\|\langle x\rangle^\alpha J^{2(1+a)}{u} _0\|_2)\\
&\leq c_t(\|u_0\|_{(5/2+a)(a+1),2}+\|\langle x\rangle^{5/2+a}u_0\|^{\frac{\alpha}{2+\alpha}}_2\|J^{(5/2+a)(a+1)}u_0\|^{\frac{2}{2+\alpha}}_2)
\end{aligned}
\end{equation}
which are all finite since $u_0\in Z_{(5/2+a)(a+1),5/2+a}$.

\end {section}


\begin{section}{Proof of Theorem \ref{theorem8}}
As it was remarked, we carry out the details for $a\in[1/2,1)$.
Thus, $7/2+a=4+\alpha$ with $\alpha=a-1/2 \in[0,1/2).$

Again, we shall use the integral equation associated to the IVP
\eqref{DGBO}
$$ u(t)=W_a(t)u_0-\int_0^tW_a(t-t')u(t')\partial_xu(t')dt',$$
 which in Fourier space reads as
 \begin{equation}
\widehat{u}(\xi,t)=e^{-it|\xi|^{1+a}\xi}\widehat{u}_0(\xi)-\frac{i}{2}\int_0^t e^{-i(t-t')|\xi|^{1+a}\xi}\xi\widehat{u^2}(\xi, t') dt'.
\label{duh2}
\end{equation}

With the notation introduced in \eqref{notation2}, we have that

\begin{equation}
\begin{aligned}
\partial_{\xi}^4\widehat{u}(\xi,t)&=F_4(t,\xi,\widehat{u}_0)-\frac{i}{2}\int_0^tF_4(t-t',\xi,\xi \widehat{u^2}(\xi,t'))\\
&=\sum_{1}^{11}E_j(t,\xi,\widehat{u}_0)-\frac{i}{2}\int_0^t\sum_{1}^{11}E_j(t-t',\xi,\xi \widehat{u^2}) dt'.
\end{aligned}
\label{der4duh}
\end{equation}

By hypothesis we have that
\begin{equation}
u\in C([-T,T]:\dot{Z}_{(1+a)(\frac{7}{2}+a)+\frac{1-a}{2},\frac{7}{2}+a-\epsilon}),\,\, \text{for some}\,\, 0<\epsilon<\!\!1.
\label{t8hyp}
\end{equation}

Therefore
\begin{equation}
u\partial_x u\in C([-T,T]:Z_{(1+a)(\frac{7}{2}+a)-\frac{1+a}{2},6+2a-2\epsilon}),
\label{t8nl1}
\end{equation}
and by using  Proposition \ref{prop3a}
\begin{equation} \label{t8smooth}
u\partial_x u\in L^1([-T,T]:H^{s_0}({\R})),\;\;\;\;\;s_0\in (0,(1+a)(7/2+a)).
\end{equation}

 In Fourier space these last two properties are

 \begin{equation}
 \widehat{u}\in C([-T,T]:{Z}_{\frac{7}{2}+a-\epsilon,(1+a)(\frac{7}{2}+a)+\frac{1-a}{2}}),
 \end{equation}
and
  \begin{equation}
  \xi\, \widehat{u}\ast\widehat{u}\in C([-T,T]:{Z}_{6+2a-2\epsilon,(1+a)(\frac{7}{2}+a)-\frac{1+a}{2}}).
\end{equation}
Also for $j=1,2,3$ one has that
\begin{equation}
u(\cdot,t_j)\in \dot{Z}_{(7/2+a)(1+a),7/2+a}\;\;\;\;\text{and so}\;\;\;\;\widehat{u}(\cdot,t_j)\in \dot{Z}_{7/2+a,(7/2+a)(1+a)}.
\label{t8hyp2}
\end{equation}
 We observe that from the equation in \eqref{DGBO} it follows that
\begin{equation}
 \frac{d\;}{dt}\int_{-\infty}^{\infty} x u(x,t) dx = \frac{1}{2}\,\|u(t)\|_2^2=\frac{1}{2}\,\|u_0\|_2^2,
 \label{moment1}
 \end{equation}
  and hence
 \begin{equation}
 \int_{-\infty}^{\infty} x u(x,t) dx =\int_{-\infty}^{\infty} x u_0(x) dx +\frac{t}{2}\,\|u_0\|_2^2
 \label{moment2}
 \end{equation}
so the first momentum of a non-null solution is a strictly increasing function of $t.$ If we prove
that there exist $\tilde{t}_1\in(t_1,t_2)$ and $\tilde{t}_2\in(t_2,t_3)$ such that for $j=1,2$
 \begin{equation}
 \int_{-\infty}^{\infty} x u(x,\tilde{t}_j) dx=0,
 \end{equation}
it will follow that $\|u_0\|_2=0,$ and therefore $u\equiv 0.$

So we just need to show that using the hypothesis in \eqref{t8hyp2} and \eqref{moment1} for $j=1,2$
there exists $\tilde{t}_1\in(t_1,t_2)$ such that
\begin{equation}
 \int_{-\infty}^{\infty} x u(x,\tilde{t}_1) dx=0.
 \end{equation}

Without loss of generality we assume that $t_1=0<t_2<t_3.$ Then, going back to equation \eqref{der4duh}
we observe that

\begin{equation}
 \begin{aligned}
 E_1=E_1(t,\xi,\widehat{u}_0)&=c_1 t|\xi|^{a-2}\text{sgn}(\xi)e^{-it|\xi|^{1+a}\xi}\widehat{u}_0(\xi)\\
 &=c_1 t|\xi|^{a-2}\text{sgn}(\xi)e^{-it|\xi|^{1+a}\xi}\widehat{u}_0(\xi)(\chi(\xi)+1-\chi(\xi))\\
 &=E_{1,1}+E_{1,2},
 \end{aligned}
 \label{t81}
 \end{equation}
with $E_{1,2}\in H^1(\R)$ for any $t\in\R$. On the other hand, by using Taylor's formula and the fact that $\widehat{u}_0(0)=0$, we obtain
\begin{equation}
\widehat{u}_0(\xi)=\xi\partial_{\xi}\widehat{u}_0(0)
+\int_0^{\xi}(\xi-\theta)\partial_{\xi} ^2\widehat{u}(\theta)d\theta
\equiv \xi\partial_{\xi}\widehat{u}_0(0)+R_2(\xi).
\end{equation}

Therefore we can write $E_{1,1}$ as
\begin{equation}
 \begin{aligned}
 E_{1,1}(t,\xi,\widehat{u}_0)&=c_1 t|\xi|^{a-2}\text{sgn}(\xi)e^{-it|\xi|^{1+a}\xi}\chi(\xi)\big(\xi\partial_{\xi}\widehat{u}_0(0)+R_2(\xi)\big)\\
 &=c_1 t|\xi|^{a-1}e^{-it|\xi|^{1+a}\xi}\chi(\xi)\partial_{\xi}\widehat{u}_0(0)+\tilde{R}_2(\xi,t).
 \end{aligned}
 \label{t82}
 \end{equation}

Let us see that for any $t\in\R$, $\tilde{R}_2(\xi,t)\in H^1(\R).$ Thus
$$
\aligned
\|\tilde{R}_2(\cdot,t)\|_2&\leq c_t\||\xi|^{a-2}\chi(\xi)|\xi|^2\|\partial_\xi ^2\widehat{u}_0\|_\infty\|_2\\
&\leq c_t \|\widehat{x^2u_0}\|_\infty\;\leq c_t \|\ji x\jd^{\frac{5}{2}^+}u_0\|_2.
 \endaligned
 $$
Since $a\in (1/2,1)$ (so $\,a-1>-1/2$)
\begin{equation}
\begin{aligned}
\|\partial_\xi \tilde{R}_2(.,t)\|_2&\leq c_t(\||\xi|^{a-3}\chi(\xi)|\xi|^2\|\partial_\xi ^2\widehat{u}_0\|_\infty\|_2\\
&\;\;+\||\xi|^{a-2}\delta(0) \chi(\xi)|\xi|^2\|\partial_\xi^2\widehat{u}_0\|_\infty\|_2\\
&\;\;+\||\xi|^{2a+1}\chi(\xi)|\|\partial_\xi^2\widehat{u}_0\|_\infty\|_2\\
&\;\;+\||\xi|^{a-2}\chi(\xi)|\xi|\|\partial_\xi ^2\widehat{u}_0\|_\infty\|_2)\\
&\leq c_t \|\langle x\rangle^{\frac{5}{2}^+}u_0\|_{2}.
\end{aligned}
\label{t83}
\end{equation}
Next we observe that
 \begin{equation}
\begin{aligned}
  &t|\xi|^{a-1}\text{sgn}(\xi)e^{-it|\xi|^{1+a}\xi}\chi(\xi)\partial_{\xi}\widehat{u}_0(0)\\
 &=t|\xi|^{a-1}\text{sgn}(\xi)\chi(\xi)\partial_{\xi}\widehat{u}_0(0)\big(1+(e^{-it|\xi|^{1+a}\xi}-1)\big)\\
 &=t|\xi|^{a-1}\text{sgn}(\xi)\chi(\xi)\partial_{\xi}\widehat{u}_0(0)+Q_2(t,\xi)
\label{t84}
\end{aligned}
\end{equation}
with
\begin{equation}
Q_2(t,\cdot)\in H^1(\R),
\label{t85}
\end{equation}
for all $t\in\R$, which follows by combining  the estimate
$|e^{i\tau}-1|\leq|\tau| $ and the fact that $ a\in (1/2,1).$

Gathering the information from \eqref{t81} to \eqref{t85} we can conclude
\begin{equation}
E_1-c_1t|\xi|^{a-1}\text{sgn}(\xi)\chi(\xi)\partial_{\xi}\widehat{u}_0(0) \in H^1(\R),
\label{t86}
\end{equation}
for all $t\in\R$, with $c_1=-(2+a)(1+a)a(a-1)i$.

Combining the  above argument and \eqref{t8nl1} we also conclude that
$$
\int_0^tE_1(t-t',\xi,\xi \widehat{u^2}(\xi,t'))dt'-c_1\int_0^t(t-t')|\xi|^{a-1}\chi(\xi) \partial_\xi(\xi\widehat{u^2})(0,t'))dt'\in H^1(\R)
$$
for all $t\in\R.$

Now we shall rewrite the terms $E_5$'s in \eqref{notation2} as
\begin{equation}\label{t87}
 \begin{aligned}
 E_5=E_5(t,\xi,&\widehat{u}_0)=c_5 t|\xi|^{a-1}\text{sgn}(\xi)e^{-it|\xi|^{1+a}\xi}\partial_\xi\widehat{u}_0(\xi)\\
 &=c_1 t|\xi|^{a-1}\text{sgn}(\xi)e^{-it|\xi|^{1+a}\xi}\partial_\xi\widehat{u}_0(\xi)(\chi(\xi)+1-\chi(\xi))\\
 &=E_{5,1}+E_{5,2}.
 \end{aligned}
 \end{equation}
with $E_{5,2}\in H^1(\R)$ for any $t\in\R$. In fact
\begin{equation}
 \begin{aligned}
 \|E_{5,2}\|_{1,2}&\leq c_t(\|D^{2a}(xu_0)\|_2+\|\langle x\rangle ^2 u_0\|_{2})\\
 &\leq c_t(\|xu_0\|_{2,2}+\|\langle x\rangle ^2 u_0\|_{2})\\
 &\leq c_t(\|\partial_x u_0\|_2+\|\langle x\rangle \partial_x^2 u_0\|_2+\|\langle x\rangle ^2 u_0\|_{2}).
  \end{aligned}
 \label{t851}
 \end{equation}

Also
\begin{equation}
 \begin{aligned}
E_{5,1}=&\, c_5 t|\xi|^{a-1}\text{sgn}(\xi)e^{-it|\xi|^{1+a}\xi}\chi(\xi)\partial_\xi\widehat{u}_0(0)\\
&+c_5 t|\xi|^{a-1}\text{sgn}(\xi)e^{-it|\xi|^{1+a}\xi}\chi(\xi)(\partial_\xi\widehat{u}_0(\xi)-\partial_\xi\widehat{u}_0(0)).
  \end{aligned}
 \end{equation}
An argument similar to that one in \eqref{t82}-\eqref{t83} shows that
\begin{equation}
c_5 t|\xi|^{a-1}\text{sgn}(\xi)e^{-it|\xi|^{1+a}\xi}\chi(\xi)(\partial_\xi\widehat{u}_0(\xi)-\partial_\xi\widehat{u}_0(0))\in H^1(\R),
\end{equation}
for all $t\in \R$. Now we consider
\begin{equation}\label{t852}
\begin{split}
c_5 t|\xi|^{a-1}\text{sgn}(\xi)&e^{-it|\xi|^{1+a}\xi}\chi(\xi)\partial_\xi\widehat{u}_0(0)
= c_5\,t|\xi|^{a-1}\text{sgn}(\xi)\chi(\xi)\partial_\xi\widehat{u}_0(0)\\
&+c_5\,t|\xi|^{a-1}\text{sgn}(\xi)(e^{-it|\xi|^{1+a}\xi}-1)\chi(\xi)\partial_\xi\widehat{u}_0(0).
\end{split}
\end{equation}
The argument in \eqref{t84} and \eqref{t85} show that
\begin{equation}
t|\xi|^{a-1}\text{sgn}(\xi)(e^{-it|\xi|^{1+a}\xi}-1)\chi(\xi)\partial_\xi\widehat{u}_0(0)\in H^1(\R),
\end{equation}
for all $t\in\R$. Hence gathering the information from  \eqref{t87} to \eqref{t852}
\begin{equation}
E_5(t,\xi,\widehat{u}_0)-c_5 t|\xi|^{a-1}\text{sgn}(\xi)\chi(\xi)\partial_\xi\widehat{u}_0(0)\in H^1(\R)
\end{equation}
for all $t\in\R$, with $c_5=-4i(2+a)(1+a)a$.

The  above argument and \eqref{t8nl1} show that
$$
\int_0^tE_5(t-t',\xi,\xi \widehat{u^2}(\xi,t'))-c_5\int_0^t(t-t')|\xi|^{a-1}\chi(\xi) \partial_\xi(\xi\widehat{u^2})(0,t'))\in H^1(\R),
$$
for all $t\in\R$. We claim that for all $t\in\R$,
\begin{equation}
E_2(t,\cdot,\widehat{u}_0),\,\, E_3(t,\cdot,\widehat{u}_0)\in H^1(\R).
  \label{t583}
 \end{equation}

It suffices to consider $E_3.$ So

\begin{equation}
\|E_3\|_2\leq c_t \|u_0\|_{3a+2,2},
\end{equation}
and
\begin{equation}
 \begin{aligned}
\|\partial_\xi E_3\|_2&\leq c_t (\|\langle \xi\rangle^{4a+3}\widehat{u}_0\|_{2}+\|\langle \xi\rangle ^{3a+2}\partial_\xi \widehat{u}_0\|_2)\\
&\leq c_t (\|\langle \xi\rangle ^{4a+3}\widehat{u}_0\|_{2}+\|\partial_\xi(\langle \xi\rangle ^{3a+2} \widehat{u}_0)\|_2)\\
&\leq c_t(\|u_0\|_{4a+3,2}+\|J_\xi(\langle \xi\rangle ^{3a+2} \widehat{u}_0)\|_2)\\
&\leq c_t(\|u_0\|_{4a+3,2}+\|J_\xi ^4 \widehat{u}_0\|^{\frac{1}{4}}_2\|(\langle \xi\rangle ^{4a+\frac83} \widehat{u}_0)\|^{\frac{3}{4}}_2)\\
&\leq c_t( \|u_0\|_{4a+3,2}+\|\langle x\rangle ^{4} u_0\|^{\frac{1}{4}}_2\|J^{4a+\frac83} u_0\|^\frac{3}{4}_2).
 \end{aligned}
\end{equation}

Since $4a+\frac83\leq (\frac{7}{2}+a)(1+a),$ and $4\leq \frac{7}{2}+a$, the claim is proved.

A similar argument and \eqref{t8smooth} show that for $j=2,3$
\begin{equation}
\int_0^tE_j(t-t',\xi, \xi \widehat{u^2}(\xi,t'))dt'\in H^1(\R),
\end{equation}
for all $t\in\R$.  Let us see now that for $j=6,9,$
\begin{equation}
E_j(t,\cdot,\widehat{u}_0)\in H^1(\R),
\end{equation}
for all $t\in\R$. In both cases the proof is similar so we just consider the case $j=9$. So

\begin{equation}
 \begin{aligned}
\|E_9\|_2&\leq c_t \|\langle \xi\rangle ^{a} \partial_\xi^2\widehat{u}_0\|_2\\
&\leq c_t(\|\langle \xi\rangle ^{a} J_\xi^2\widehat{u}_0\|_2+\|\langle \xi\rangle^{a} \widehat{u}_0\|_2)\\
&\leq c_t(\|J^{a} \langle x\rangle ^{2}u_0\|_2+\|J^{a} {u}_0\|_2),
 \end{aligned}
\end{equation}
and
\begin{equation}
 \begin{aligned}
\|\partial_\xi& E_9\|_2\leq c_t ( \||\xi|^{a-1} \partial_\xi^2\widehat{u}_0\|_2
+\|\langle \xi\rangle ^{1+2a} \partial_\xi^2\widehat{u}_0\|_2\|\langle \xi\rangle ^{a} \partial_\xi^3\widehat{u}_0\|_2)\\
&\leq c_t(\| \partial_\xi^2\widehat{u}_0\|_\infty+\|\langle \xi\rangle ^{1+2a} J_\xi^2\widehat{u}_0\|_2\|J_\xi^3\langle \xi\rangle^{a}\widehat{u}_0\|_2
+\|\langle \xi\rangle ^{1+2a}\widehat{u}_0\|_2)\\
&\leq c_t(\| x^2{u}_0\|_{1,2}+\|J^{2a+1} \langle x\rangle ^{2}u_0\|_2+\|\langle x\rangle ^{3}J^{a} {u}_0\|_2+\| {u}_0\|_{2a+1,2})\\
&\leq c_t(\| x^2{u}_0\|_{2}+\|J^{2a+1} \langle x\rangle ^{2}u_0\|_2+\|\langle x\rangle ^{3}J^{a} {u}_0\|_2+\| {u}_0\|_{2a+1,2})),
 \end{aligned}
\end{equation}
which can be bounded by interpolation as well.

Also, a similar argument and \eqref{t8smooth} shows again that for $j=6,9$
\begin{equation}
\int_0^tE_j(t-t',\xi, \xi \widehat{u^2}(\xi,t'))dt'\in H^1(\R),
\end{equation}
for all $t\in\R$.  By hypotheses \eqref{t8hyp}, \eqref{t8smooth} and \eqref{t8hyp2} we have

\begin{claim2}
Let $\alpha=a-\frac12\in(0,\frac12)$ and $j=4,7,8$ and $11$. Then
\begin{equation}
E_j(t-t',\xi,\partial_\xi (\xi \widehat{u^2})), \,\,\int_0^tE_j(t-t',\xi \widehat{u^2})(\xi,t')dt'\in H^\alpha(\R),
\label{claim2}
\end{equation}
for all $t\in\R.$
\end{claim2}

{\it Proof of Claim 2}:
Due to the form of the terms, by interpolation  it suffices  to consider the cases $j=4$ and $j=11$. Thus,
\begin{equation}
\|E_4\|_2\leq c_t \||\xi|^{4(1+a)}\widehat{u}_0\|_{2}\leq c_t \|u_0\|_{4(1+a),2},
\label{t8E4y11}\end{equation}
and
\begin{equation}
\|E_{11}\|_2\leq c_t \|\partial_\xi^{4}\widehat{u}_0\|_{2}\leq c_t \|\langle \xi\rangle ^{4}u_0\|_{2},
\end{equation}
and hence both quantities are finite. Now
\begin{equation*}
 \begin{aligned}
\|D^\alpha_\xi E_4\|_2&\leq c_t( \|
|\xi|^{1+a}\widehat{u}_0\|_{2}+\|
|\xi|^{(4+\alpha)(1+a)}\widehat{u}_0\|_{2}
+\|\mathcal D^\alpha_\xi( |\xi|^{4(1+a)}\widehat{u}_0)\|_{2})\\
&\leq c_t(\|\langle \xi\rangle^{(4+\alpha)(1+a)}\widehat{u}_0\|_{2}
+\|\mathcal D^\alpha_\xi( |\xi|^{4(1+a)}\widehat{u}_0)\|_{2})\\
&\leq c_t(\| {u}_0\|_{(4+\alpha)(1+a),2}+\|\mathcal D^\alpha_\xi(
|\xi|^{4(1+a)}\widehat{u}_0)\|_{2}),
 \end{aligned}
\end{equation*}
but
\begin{equation}
 \begin{aligned}
\|\mathcal D^\alpha _\xi( &|\xi|^{4(1+a)}\widehat{u}_0)\|_{2}\leq \|\mathcal D_\xi ^{\alpha}(\frac{|\xi|^{4(a+1)}}{\langle \xi\rangle ^{4(1+a)}}{\langle \xi\rangle}^{4(1+a)} \widehat{u} _0)\|_2\\
&\leq c(\|{\langle \xi\rangle}^{4(1+a)} \widehat{u} _0\|_{2}
+\|\mathcal{D}_\xi ^{\alpha} (\langle \xi\rangle ^{4(1+a)}\widehat{u}_0)\|_2)\\
&\leq c(\|J^{4(1+a)}u_0\|_{2}+\|J_\xi ^{\alpha}(\langle \xi\rangle^{4(1+a)}{u} _0)\|_2)\\
&\leq  c(\|u_0\|_{4(1+a),2}+\|\langle x\rangle^\alpha J^{4(1+a)}{u} _0\|_2)\\
&\leq c(\|u_0\|_{4(1+a),2}+\|\langle x\rangle^{4+\alpha}u_0\|^{\frac{\alpha}{4+\alpha}}_2\|J^{(4+\alpha)(a+1)}u_0\|^{\frac{4}{4+\alpha}}_2),
\end{aligned}
\end{equation}
therefore
\begin{equation}
\|D^\alpha_\xi E_4\|_2\leq c(\|u_0\|_{(4+\alpha)(1+a),2}+\|\langle x\rangle^{4+\alpha}u_0\|_2),
\end{equation}
which is finite by the hypothesis at $t_1=0$. Also
\begin{equation}\label{t8E4y11last}
 \begin{aligned}
\|D^\alpha _\xi &E_{11}\|_2\leq c_t(
\|\partial_\xi^{4}\widehat{u}_0\|_{2}
+\|\xi|^{\alpha(1+a)}\partial_\xi^4\widehat{u}_0\|_{2}
+\|\mathcal D^\alpha _\xi( \partial_\xi^4\widehat{u}_0)\|_{2})\\
&\leq c_t(\|\langle x\rangle^{4}u_0\|_{2}+\|\langle \xi\rangle^{\alpha(1+a)}\partial_\xi^4\widehat{u}_0\|_{2}
+\||x|^{4+\alpha}u_0\|_{2})\\
&\leq c_t(\|J^{\alpha(1+a)}(x^4u_0)\|_{2}+\|\langle x\rangle^{4+\alpha}u_0\|_{2})\\
&\leq c_t(\|J^{\alpha(1+a)}(\langle x\rangle^{4}u_0)\|_{2}+\|J^{\alpha(1+a)}(\langle x\rangle^{2}u_0)\|_{2}\\
&\;\;\;\;+\|J^{\alpha(1+a)}u_0\|_{2}
+\|\langle x\rangle^{4+\alpha}u_0\|_{2})\\
&\leq c_t(\|\langle
x\rangle^{4+\alpha}u_0\|^{\frac{4}{4+\alpha}}_2\|J^{(4+\alpha)(a+1)}u_0\|^{\frac{\alpha}{4+\alpha}}_2\\
&\;\;\;\;+\|\langle x\rangle^{4}u_0\|^{\frac{1}{2}}_2\|J^{2\alpha(a+1)}u_0\|^{\frac{1}{2}}_2\\
&\;\;\;\;+\|J^{\alpha(1+a)}u_0\|_{2}
+\|\langle x\rangle^{4+\alpha}u_0\|_{2})\\
&\leq c_t(\|u_0\|_{(4+\alpha)(1+a),2}+\|\langle x\rangle^{4+\alpha}u_0\|_{2}).
\end{aligned}
\end{equation}

Combining \eqref{t8E4y11}--\eqref{t8E4y11last} and \eqref{t8smooth} it follows that
\begin{equation}
\int_0^tE_j(t-t',\xi,\xi \widehat{u^2})(\xi,t')dt'\in H^{\alpha}(\R),
\end{equation}
for all $t\in\R$ for $j=4, 11,$ and consequently for $j=7,8,9$ as well.

Summing up, we can conclude that
\begin{equation}\label{cond5}
\begin{split}
&D_\xi^{\alpha}\partial_{\xi}^4\widehat{u}(\cdot,t)\in L^2(\R)\text{\hskip10pt if and only if}\\
&D_\xi^{\alpha}\Big(t|\xi|^{a-1}\chi(\xi)\partial_\xi\widehat{u}_0(0)\!-\!\!\int_0^t\!\!(t-t')
|\xi|^{a-1}\chi(\xi)\partial_\xi\widehat{u^2}(0,t')dt'\Big)\in L^2(\R),
\end{split}
\end{equation}
for any fixed $t\in\R$. We observe that
\begin{equation}
\partial_\xi \widehat{u}_0(0)=\widehat{-ixu_0}(0)=-i\int xu_0(x)\,dx,
\label{t8mom1}
\end{equation}
and by \eqref{moment1}
\begin{equation}
\begin{split}
\partial_\xi (i\frac{\xi}{2}\widehat{u}^2)(0,t')&=\widehat{-ixu\partial_x u}(0,t')=-i\int xu\partial_xu(x,t')\,dx,\\
&=\frac{i}{2}\|u(t')\|^2_2=\frac{i}{2}\|u(0)\|^2_2=i\frac{d}{dt}\int xu(x,t) dx,
\end{split}
\end{equation}
and by integration by parts
\begin{equation}\label{t8ibyp}
 \begin{split}
 &t_2\partial_\xi \widehat{u}_0(0)-\frac{i}{2}\int_0^{t_2}\,(t_2-t')\partial_\xi (\xi\widehat{u}^2)(0,t')dt'\\
 &=-it_2\int xu_0\,dx\!-\! i\int_0^{t_2}\,(t_2-t')\frac{d}{dt'}\int xu(x,t') dx\,dt'\\
 &=-it_2\int xu_0(x)dx- i(t_2-t')\int xu(x,t')dx\Big|_{t'=0}^{t'=t_2}\\
 &\;\;\;-\!i\int_0^{t_2}\!\!\int xu(x,t') dxdt'\\
 &=-i\int_0^{t_2}\,\int xu(x,t') dx\,dt'.
 \end{split}
\end{equation}

Thus from our hypothesis at $t=t_2,$ it follows that
\begin{equation}\label{t8conclus}
D_\xi^{\alpha}(\chi(\xi)|\xi|^{a-1})\int_0^{t_2}\,\int xu(x,t') dx\,dt'\in L^2(\R),
\end{equation}
but we recall that $\alpha=a-1/2$ and from Proposition \ref{prop3}

\begin{equation*}
D_\xi^{\alpha}(\chi(\xi)|\xi|^{a-1})\approx \mathcal D_\xi^{\alpha}(\chi(\xi)|\xi|^{\alpha-1/2})
\notin L^2(\R)\,\,\text{if} \,\,\alpha\in(0,1).
\end{equation*}
Therefore for \eqref{t8conclus} to hold is necessary that
 \begin{equation}
 \int_0^{t_2}\,\int xu(x,t') dx\,dt'=0,
 \label{t8conclus2}
 \end{equation}
 and since $F(t)=\int xu(x,t) dx$ is a continuous function, there exists $\tilde{t}_1\in (0,t_2)$
 where $F(t)$ must vanish and this completes the proof.
 \end{section}


 \begin{section}{Proof of Theorem \ref{theorem9}}\label{proof9}
 Without loss of generality we can assume that
 \begin{equation}
 t_1=0 \qquad\text{and}\qquad \int xu_0(x) dx=0.
 \end{equation}

Thus in this case, combining \eqref{t8mom1}, \eqref{t8ibyp}-\eqref{t8conclus2} and \eqref{moment2},
we have for $t_2\neq 0$ that

\begin{equation}\label{}
\begin{aligned}
 &\mathcal D_\xi^{\alpha}\partial_{\xi}^4\widehat{u}(\cdot,t)\in L^{2}(\R), \,\,\,\,\,\text {if and only if}\\
&\int_0^{t_2}\,\int xu(x,t') dx\,dt'=0, \,\,\,\,\,\text {if and only if}\\
&\int_0^{t_2}\frac12 t'\|u_0\|_2^2 dt'=\frac{t_2^2}{4}\|u_0\|_2^2=0,\,\,\,\,\,\text {if and only if}\\
& \|u_0\|_2^2=0\,\,\,\,\text {if and only if}\,\,\,\, u_0=0.
\end{aligned}
\end{equation}
\end{section}


\begin{section}{Proof of Theorem \ref{theorem10}}\label{proof10}
We shall consider only the case $a\in[1/2,1),$ so that $\tilde a=1.$

From the argument of the proof in Theorem \ref{theorem8}, with $\alpha=1/2$ in \eqref{t8E4y11}--\eqref{t8E4y11last}
and \eqref{t8smooth}, we can conclude from our hypothesis $s\geq (\frac72+a)(1+a)+\frac{1-a}{2},$ that
for $t\neq 0$

\begin{equation*}\label{}
\begin{aligned}
 D_\xi^{1/2}\partial_{\xi}^4\widehat{u}(\cdot,t)\in L^{2}(\R), \,\,\,\,\,&\text{if and only if}\\
D_\xi^{1/2}\Big(t|\xi|^{a-1}\chi(\xi)\partial_\xi \widehat{u}_0(0)-&\int_0^t(t-t')|\xi|^{a-1}\chi(\xi)\partial_\xi(\xi\widehat{u^2}(0,t))dt'\Big)\in L^2(\R),\\
&\text{if and only if}\\
D_\xi^{1/2}(\chi(\xi)|\xi|^{a-1})&\int_0^{t}\,\int xu(x,t') dx\,dt'\in L^2(\R),\\
&\text{if and only if}\\
\int_0^{t}\,\int xu(x,t') dx\,dt'&=\int_0^{t}(\int xu_0(x)\,dx+\frac12 t'\|u_0\|_2^2) dt'\\
&=t(\int xu_0(x)\,dx+\frac14 t\|u_0\|_2^2)=0.
\end{aligned}
\end{equation*}
\end{section}


\end{document}